\documentclass[11pt,a4paper]{article}
\usepackage[body={154mm,229mm},top=34mm]{geometry}
\usepackage{amscd}
\usepackage{mathrsfs}
\usepackage{epsfig}
\usepackage{amsmath}
\usepackage{amssymb}
\usepackage{bm}
\usepackage{amsthm}
\usepackage{epsfig}
\usepackage{verbatim}
\usepackage{url}
\usepackage{color}
\usepackage{hyperref}
\include{pspicture}

\newcommand{\ninseps}[3]{
\begin{figure}[h]
 \scalebox{#3}{\includegraphics{#1}}

\vspace{-0.65cm}
\caption{\hspace{0.25cm}#2\label{f:#1}}
\end{figure}
}

\newcommand{\ninsepsc}[3]{
\begin{figure}[h]
\begin{center}
 \scalebox{#3}{\includegraphics{#1}}
\end{center}

\vspace{-0.65cm}
\caption{\hspace{0.25cm}#2\label{f:#1}}
\end{figure}
}

\newtheorem{theorem}{Theorem}[section]
\newtheorem{lemma}{Lemma}

\newtheorem{definition}{Definition}[section]
\newtheorem{remark}{Remark}

\newtheorem{proposition}{Proposition}[section]

\usepackage{makeidx}
\usepackage{mathabx}
\usepackage{mathrsfs} 
\begin{document}
\title{Hitting probabilities of constrained random walks representing tandem networks}
\author{Ali Devin Sezer\footnote{Middle East Technical University, Institute of Applied Mathematics, Ankara, Turkey, devinsezer@gmail.com.
Sections 3,4 and 5 of the present work are revisions of Sections 5,6 and subsection 8.3 of \cite{sezer2015exit}}}
\maketitle
\begin{abstract}
Let $X$ be the constrained random walk on $\mathbb{Z}_+^d$ $d >2$, having increments $e_1$, $-e_i+e_{i+1}$ $i=1,2,3,...,d-1$ and $-e_d$ with probabilities $\lambda$, $\mu_1$, $\mu_2$,...,$\mu_d$, where $\{e_1,e_2,..,e_d\}$ are the standard basis vectors.  The process $X$ is assumed stable, i.e., $\lambda < \mu_i$ for all $i=1,2,3,...,d.$ Let $\tau_n$ be the first time the sum of the components of $X$ equals $n$. We derive approximation formulas for the probability ${\mathbb P}_x(\tau_n < \tau_0)$.  For $x \in  \bigcup_{i=1}^d \Big\{x \in {\mathbb R}^d_+: \sum_{j=1}^{i} x(j)$ $>  \left(1 - \frac{\log \lambda/\min \mu_i}{\log \lambda/\mu_i}\right) \Big\}$ and a sequence of initial points $x_n/n \rightarrow x$ we show that the relative error of the approximation decays exponentially in $n$.  The approximation formula is of the form ${\mathbb P}_y(\tau < \infty)$ where $\tau$ is the first time the sum of the components of a limit process $Y$ is $0$; $Y$ is the process $X$ as observed from a point on the exit boundary except that it is unconstrained in its first component (in particular $Y$ is an unstable process); $Y$ and ${\mathbb P}_y(\tau< \infty)$ arise naturally as the limit of an affine transformation of $X$ and the probability ${\mathbb P}_x(\tau_n < \tau_0).$ The analysis of the relative error is based on a new construction of supermartingales. We derive an explicit formula for ${\mathbb P}_y(\tau < \infty)$ in terms of the ratios $\lambda/\mu_i$ which is based on the concepts of harmonic systems and their solutions and conjugate points on a characteristic surface associated with the process $Y$; the derivation of the formula assumes $\mu_i \neq \mu_j$ for $i\neq j.$
\end{abstract}

\section{Introduction and Definitions}
For an integer $d \ge 3$
let $X$ be a random walk with independent and identically
distributed increments $\{I_1,I_2,I_3,...\}$, $I_k \in {\mathcal V} \subset {\mathbb Z}^d$,  constrained
to remain in ${\mathbb Z}_+^d$, i.e.,
\[
X_{k+1} = X_k + \pi(X_k,I_{k+1}),
\]
where
\[
\pi(x,v) \doteq \begin{cases} v, &\text{ if } x +v \in {\mathbb Z}_+^d \\
				  0      , &\text{otherwise.}
\end{cases}
\]
The constraining boundaries of $X$ are
\[
\partial_j = \{ x \in {\mathbb Z}^d: x(j) = 0 \}, j \in \{1,2,3,...d\}.
\]
For 
$ A_n \doteq \{ x \in {\mathbb Z}_+^d: \sum_{i=1}^d x(i) \le n \}$
and
$\partial A_n \doteq \{ x \in {\mathbb Z}_+^d: \sum_{i=1}^d x(i) = n \}$
define the stopping time
\[
\tau_n \doteq \inf \{k \ge 0: X_k \in \partial A_n\}.
\]
The goal of this work is to develop approximations of the hitting
probability
\[
{\mathbb P}_x( \tau_n < \tau_0) 
\]
for $x \in {\mathbb Z}_+^d$,  $x \in A_n$
when $X$ is a constrained random walk that represents $d$ queues
in tandem (a tandem network), i.e., 
when the set of possible increments of $X$
are
\begin{align*}
{\mathcal V} &= \{e_1, -e_1 + e_2,..., -e_j+e_{j+1},...,-e_{d-1}+e_d, -e_d \},\\
e_i(j) &= \begin{cases} 1, &\text{ if } i = j,\\
		       0, &\text{otherwise,}
\end{cases}
\end{align*}
$i,j=1,2,3,...,d$; $\{e_i,i=1,2,3,...,d\}$
are the unit vectors in ${\mathbb Z}^d$.
The distribution of the increments is given as follows:
\begin{align*}
{\mathbb P}(I_k = e_1) &= \lambda, \\
{\mathbb P}(I_k = e_{i+1}-e_{i}) &= \mu_i,~~~ i=1,2,3,...,d-1,\\
{\mathbb P}(I_k = -e_d) &= \mu_d.
\end{align*}
We assume $X$ to be stable:
\begin{equation}\label{as:stability}
\lambda < \max_{i=1}^d \mu_i
\end{equation}
which implies
\[
\rho_i \doteq \lambda/\mu_i, ~~\rho \doteq \max_i \rho_i < 1.
\]

The random walk $X$ can represent $d$ servers/processes working in tandem;
in this interpretation, $\lambda$ is the arrival rate to the first queue
the $\mu_i$ are the processing rates of the servers and the components
of $X$ are the number of items/packets waiting for service and the probability
$p_n(x) = {\mathbb P}_x(\tau_n < \tau_0)$ is the probability that the number
of packets in the system reaches $n$ before the system empties.
The analysis of $p_n$ goes at least back to
\cite{ GlassKou,parekh1989quick}.
Stability of $X$ implies that this probability decays exponentially in $n$.
Its exponential decay rate (i.e., the large deviations limit) is computed in
\cite{GlassKou, ignatiouk2000large} as
\[
\lim_{n\rightarrow \infty} -\frac{1}{n}\log p_n( x_n) =-\log\rho
\]
for $x=x_n$, $x_n/n \rightarrow 0$.
Because it is an exponentially decaying probability
its simulation requires variance reduction algorithms, see, e.g.,
\cite{DSW,blanchet2013optimal} and the references in these works.
The work \cite{sezer2018approximation} develops precise analytical formulas for this probability
for $d=2$ based on an affine transformation of $X$ and 
${\mathbb P}_x(\tau_n < \tau_0)$; the goal of the present work is to extend
these results to dimensions $3$ or more.
There is a wide literature on the approximation/simulation of probabilities of the type 
${\mathbb P}_x(\tau_n < \tau_0)$; we refer the reader to \cite[Sections 1,6]{sezer2018approximation}
for a literature review.

Let ${\mathcal I}_1 \in {\mathbb R}^{d\times d}$ be the diagonal matrix
with diagonal entries ${\mathcal I}_1(1,1) =-1$, ${\mathcal I}_1(j,j) =1$
$j >1$ and $T_n = n + {\mathcal I}_1.$ 
Define
\begin{equation}\label{e:defY}
J_k \doteq {\mathcal I}_1 I_k,~~~ Y_{k+1} \doteq Y_k + \pi_1(Y_k,J_{k+1}),
\end{equation}
where
\[
\pi_1(y,v) \doteq \begin{cases} v, &\text{ if } x +v \in {\mathbb Z} \times {\mathbb Z}_+^{d-1} \\
				  0      , &\text{otherwise.}
\end{cases}
\]
Define, $\Omega_Y \doteq {\mathbb Z} \times {\mathbb Z}_+^{d-1}$,
\begin{equation*}
B \doteq \left\{ y \in \Omega_Y, y(1) \ge \sum_{j=2}^d y(j) \right\};
\end{equation*}
the boundary of $B$ is 
\begin{equation*}
\partial B = \left\{ y \in \Omega_Y, y(1) = \sum_{j=2 }^d y(j) \right\}.
\end{equation*}
The limit stopping time
\begin{equation*}
\tau \doteq \inf \{k: Y_k \in \partial B\}
\end{equation*}
is the first time $Y$ hits $\partial B$.
Our goal is to approximate ${\mathbb P}_x(\tau_n < \tau_0)$
by the limit probability ${\mathbb P}_{T_n(x)}(\tau < \infty) = {\mathbb P}(
\tau < \infty | Y_0 = T_n(x)).$ The process $Y^n =T_n(X)$ is
the process $X$ as observed from the boundary
point $(n,0,...,0) \in \partial A_n.$
The limit process $Y$ and the limit probability ${\mathbb P}_y(\tau < \infty)$
is obtained by letting $n\rightarrow \infty$, see Figure \ref{f:Tn2tex}
for an illustration in two dimensions.

\begin{figure}[h]
\begin{center}
\scalebox{0.8}{
\centerline{\input{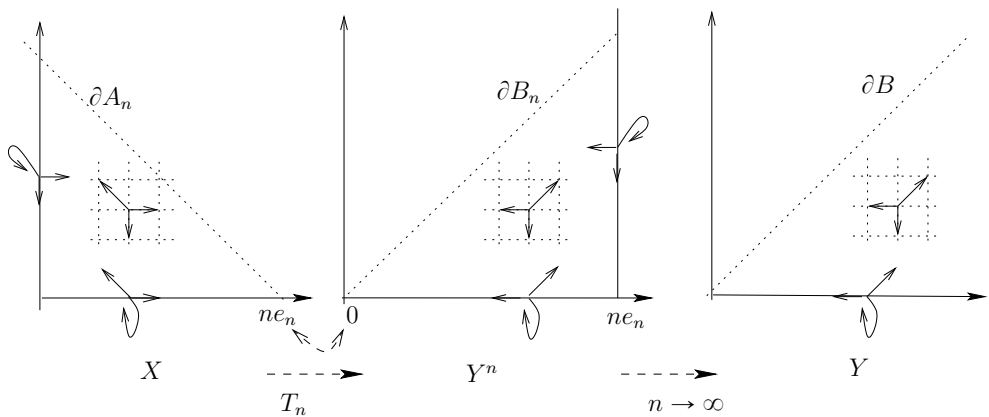}}}
\end{center}

\vspace{-0.65cm}
\caption{\hspace{0.25cm}Derivation of the limit problem via an affine transformation in two dimensions\label{f:Tn2tex}}
\end{figure}

Define
\begin{align*}
R_\rho &\doteq 
\bigcap_{i =1 }^d \left\{x \in {\mathbb R}^d_+: \sum_{j=1}^{i} x(j) \le  \left(1 - \frac{\log \rho}{\log \rho_i}\right) \right\}\\
\bar{R}_{\rho,n} &\doteq 
\bigcup_{i =1}^d \left\{x \in {\mathbb Z}^d_+: \sum_{j=1}^{i} x(j) \ge 1+ n\left(1 - \frac{\log \rho}{\log \rho_i}\right) \right\}.
\end{align*}
\[
g: {\mathbb R}_+^d \mapsto {\mathbb R}, g(x) = \max_{i \in \{1,2,..,d\}}
(1-\sum_{j=1}^i x(j))\log_{\rho}\rho_i
\]
First main result of the present paper is the following:
\begin{theorem}\label{t:main1}
 For  $\epsilon > 0$
there exists $n_0 > 0$ such that
\begin{equation}\label{e:mainrelativeerrorboundintro}
\frac{ |{\mathbb P}_{x}(\tau_n < \tau_0) - {\mathbb P}_{T_n(x)}(\tau < \infty)|}{ {\mathbb P}_{x}(\tau_n < \tau_0)} \le \rho^{n(1 - g(x/n)-\epsilon)}
\end{equation}
for all $n > n_0$ and for any $x \in \bar{R}_{\rho,n}$.
In particular, for $x_n/n \rightarrow x  \in A - R_\rho$ the relative error
decays exponentially with rate $-\log(\rho)(1-g(x))>0$, i.e.,
\begin{equation}\label{e:decayrateintro}
\liminf_n -\frac{1}{n} \log \left(
\frac{ |{\mathbb P}_{x_n}(\tau_n < \tau_0) - {\mathbb P}_{T_n(x_n)}(\tau < \infty)|}{ {\mathbb P}_{x_n}(\tau_n < \tau_0)} \right)
 \ge -\log(\rho)(1-g(x)).
\end{equation}
\end{theorem}

A precise statement corresponding to Figure \ref{f:Tn2tex} is
\cite[Proposition 1]{sezer2018approximation} 
which states
\begin{equation}\label{e:convergeiny}
\lim_n {\mathbb P}_{T_n(y)}(\tau_n < \tau_0) =
{\mathbb P}_y(\tau < \infty), y \in B,
\end{equation} 
for any stable Jackson network
in any dimension. In \eqref{e:convergeiny} the initial point
of the process is specified and fixed in $y$-coordinates;
in \eqref{e:decayrateintro} it is specified in scaled $x$ coordinates
(as is done in large deviations analysis).
For fixed $y \in B$,
the probability ${\mathbb P}_{T_n(y)}(\tau_n < \tau_0)$ doesn't
decay to $0$ in $n$
but converges to the nonzero probability ${\mathbb P}_y(\tau < \infty).$
In \eqref{e:decayrateintro},
where the initial position is fixed in scaled $x$ coordinates both
of the probabilities
${\mathbb P}_{x_n}(\tau_n < \tau_0)$ and
${\mathbb P}_{T_n(x_n)}(\tau < \infty)$ decay to $0$ exponentially.
The limit \eqref{e:decayrateintro} expresses that the difference between
them decays exponentially faster than 
${\mathbb P}_{x_n}(\tau_n < \tau_0)$.

Previous results of the type \eqref{e:decayrateintro} are as follows:
\cite[Proposition 8]{sezer2018approximation} treats the  case $d=2$, 
\cite[Theorem 6.1]{unlu2019excessive}
treats the constrained simple random walk in two dimensions
(i.e., the constrained random walk with increments $\pm e_i$, $i=1,2$)
and \cite[Theorem 6.1]{kabran2020approximation} treats 
the case $d=2$ when the dynamics of the constrained random walk
is Markov modulated (i.e., in addition to $X$ there is an additional finite
state Markov chain $M$ that determines the jump distributions of $X$).
These prior results state that the relative error on the left
side of \eqref{e:decayrateintro} converges to $0$ exponentially at a rate
depending on $x$; the precise formulation of the decay rate as in the
right side of \eqref{e:decayrateintro} is new. 
The prelimit statement \eqref{e:mainrelativeerrorboundintro} that is specified
in terms of an unscaled initial point $x$ is also new.
Theorem \ref{t:main1} uses only the stability assumption on model parameters;
all of the prior results just cited use an additional assumption, which
in the present case would be $\mu_i \neq \mu_j$ for $i \neq j$; 
dropping of this assumption in the error analysis arises from a 
significant change in the argument and we comment on it below
(see the third paragraph below).
We will assume $\mu_i \neq \mu_j$ for $i \neq j$ in Theorem \ref{p:exactformulaDtandem}
which gives an explicit formula for ${\mathbb P}_y(\tau < \infty).$

As in the works just cited we will use the following idea in our
analysis of the relative error:
because the dynamics of $X$ and $Y$ differ only on 
$\partial_1$, 
the events $\{\tau_n < \tau_0\}$ and $\{\tau < \infty\}$
mostly overlap. This is proved as follows: 1) find an event
containing the difference of the events
 $\{\tau_n < \tau_0\}$ and $\{\tau < \infty\}$ 2) prove that the
upperbound event has a small probability. 
Let $\bar\tau_0 =  \inf\left\{k \ge 0: \sum_{j=1}^d Y_k(j) = n\right\} = \inf\left\{k \ge 0: \sum_{j=1}^n T_n(Y_k(j)) = 0 \right\}$ and
let $\sigma_{j,j+1}$ be the first time $X$ hits $\partial_{j+1}$ after hitting
$\partial_j$ (see \eqref{d:sigma01} for the precise definition).
Lemmas \ref{l:beforesigmadd1}, \ref{l:comparetau} and \ref{l:boundondifference}
show that the difference between 
$\{\tau_n < \tau_0\}$ and $\{\tau < \infty\}$
is contained in the
union of $\{\bar\tau_0 < \tau < \infty\}$ and 
$\{\sigma_{d-1,d} < \tau_n < \tau_0\}$.
The proofs of these lemmas 
are more complex compared to their counterparts 
\cite{sezer2018approximation,kabran2020approximation} because $d$ is
now  arbitrary.
Some of the novelties are explained in the paragraphs below.

A function $f:\Omega_Y \mapsto {\mathbb R}$ is said to be 
$Y$-(sub/super)harmonic  
if
 $f(y) = (\le ~/ \ge) {\mathbb E}_y[f(Y_1)] $
for all $y \in \Omega_Y$.
Note that $y\mapsto {\mathbb P}_y(\tau < \infty ) ={\mathbb P}(\tau < \infty | Y_0 = y)$ is a $Y$-harmonic function.
An upperbound on the probability
${\mathbb P}_y(\bar\tau_0 < \tau < \infty)$
follows from the Markov property of $Y$
and an upperbound on 
${\mathbb P}_y(\tau < \infty )$;
in subsection \ref{ss:tauy} we construct an upperbound for this probability.
In previous works treating two dimensions the upperbound follows
directly from the computation/approximation of ${\mathbb P}_y(\tau < \infty)$:
in \cite{sezer2018approximation} there is a simple explicit formula for ${\mathbb P}_y(\tau <\infty)$ and in \cite{kabran2020approximation} an upperbound can be constructed in terms
the $Y$-harmonic functions used in the computation of
${\mathbb P}_y(\tau < \infty).$
In the present setup the formula for ${\mathbb P}_y(\tau < \infty)$ is
more complex (see Theorem \ref{p:exactformulaDtandem}); 
instead of it, in subsection \ref{ss:tauy} we construct
simpler $Y$-superharmonic functions (and corresponding supermartingales) 
that imply the bound we seek
on ${\mathbb P}_y(\tau < \infty)$ (we comment on this further in the next paragraph).
Subsection \ref{ss:sigmad1d} derives an upperbound on the probability
${\mathbb P}_x(\sigma_{d-1,d} < \tau_n < \tau_0)$
(Proposition \ref{p:upperboundonlong}). For the proof we construct
a supermartingale corresponding to the event 
$\{\sigma_{d-1,d} < \tau_n < \tau_0\}$. The event happens in $d$
stages (the process moves from stage $j$ to $j+1$ upon hitting
$\partial_{j+1}$); the supermartingale is obtained by applying one
of the $Y$-superharmonic functions of subsection \ref{ss:tauy} to $X$ at each stage.
Because $Y$ has unconstrained dynamics on $\partial_1$, the process resulting
from the application of these functions is not a supermartingale on $\partial_1$;
to compensate for this we add a strictly decreasing term to the resulting process (see \eqref{e:defsupS}).
As in previous works \cite{sezer2018approximation,unlu2019excessive,kabran2020approximation} we truncate time 
to manage this additional term (see \eqref{e:decomposesup}).

For $\beta \in {\mathbb C}$ and $\alpha \in {\mathbb C}^{\{2,3,4...,d\}}$ 
define the function $[(\beta,\alpha),\cdot]:{\mathbb Z}^d \mapsto {\mathbb C}$
as
\[
[ (\beta,\alpha) , y  ] = \beta^{y(d)-\sum_{j=2}^n y(j)} \prod_{j=2}^d \alpha(j)^{y(j)}.
\]
In \cite{sezer2018approximation,unlu2019excessive,kabran2020approximation} the supermartingales in the relative error analysis are constructed from functions of the above form
where at least some of the parameters $\beta$ and $\alpha$ take values in $\{\rho_i, i=1,2\}$.
A novel feature of the analysis in the present work is the use of values for these parameters that are strictly different from $\rho_i$,
see Propositions \ref{p:hkr}  and \ref{p:Ysuperharmonicfuns}. 
This allows the construction of strictly $Y$-superharmonic functions which
have much simpler structures compared to the $Y$-harmonic functions appearing
in the computation of ${\mathbb P}_y(\tau < \infty)$. The construction
of the just mentioned 
strictly $Y$-superharmonic functions do not need any assumptions
beyond the stability assumption \eqref{as:stability}; 
this is the reason we are able to derive
\eqref{e:mainrelativeerrorboundintro} and \eqref{e:decayrateintro}
based only on the stability assumption. 

With this theorem, the approximation of ${\mathbb P}_x(\tau_n < \tau_0)$ reduces
to the computation of ${\mathbb P}_y(\tau < \infty).$ Considered as a function
$y\mapsto {\mathbb P}_y(\tau < \infty)$ of $y$, this probability is the unique
$\partial B$-determined $Y$-harmonic function taking the value $1$ on $\partial B.$
Section \ref{s:dg2} reduces the construction of these
functions to solutions of systems of equations represented by graphs 
with labeled edges
(a ``harmonic system'', see Definition
\ref{d:Yharmonic}). This reduction
can be easily carried out for constrained random walks arising from any Jackson network and
therefore in this section we will work in that generality. Section \ref{s:hstq}
introduces a class of
harmonic systems for tandem networks and provides  solutions
for them.
Theorem \ref{p:exactformulaDtandem} gives an explicit formula
for $y\mapsto {\mathbb P}_y(\tau < \infty)$ 
as a linear combination of the functions defined by these solutions.
Section \ref{s:numerical} provides numerical examples showing the effectiveness of the formulas obtained.
In Section \ref{s:conclusion} we comment on future work.

\section{Error Analysis}
The goal of this section is to prove Theorem \ref{t:main1}.
This theorem generalizes \cite[Proposition 8]{sezer2018approximation}, 
which treats $d=2$, to an arbitrary positive dimension $d > 0$.
The proof is based on the stopping times
\begin{align}\label{d:sigma01}
\sigma_{0,1} &\doteq \inf \{k \ge 0: X_k \in \partial_1\}\\
\sigma_{j-1,j} &\doteq \inf \{k > \sigma_{j-2,j-1}: X_k \in \partial_j\},  \notag
j=2,3,...,d.
\end{align}
Time $\tau_0$ is the first time the set of all components of $X$ equal $0$; this definition and the dynamics of $X$ imply
$\tau_0 \ge \sigma_{d-1,d}.$
We will use these stopping times to show that the events 
$\{tau_n < \tau_0\}$ and $\{\tau < \infty) \}$ mostly overlap. 
In what follows it will be convenient to represent $Y$ in $x$-coordinates:
$\bar{X}_k = T_n(Y_k)$,
$\bar{X}$ has the same dynamics as $X$ except on $\partial_1$ where it is not
constrained, i.e.,
\[
\bar{X}_{k+1} = \bar{X}_k + \pi_1(\bar{X}_k, I_{k+1});
\]
$\bar{X}$ and $X$ processes start from the same point:
\begin{equation}\label{e:sameinitcond}
X_0 = \bar{X}_0 = x_n;
\end{equation}
the $Y$ process then has initial point $Y_0 = T_n(x_n).$
Define
\[
\bar{\tau}_n \doteq \inf\left \{k: \sum_{j=1}^d \bar{X}_k(j) = n \right\};
\]
we note $\bar{\tau}_n = \tau$, therefore:
\[
{\mathbb P}_{x_n}(\tau_n < \infty) = {\mathbb P}_{T_n(x_n)}(\tau < \infty).
\]

Define ${\bm S} : {\mathbb Z}^d \rightarrow {\mathbb Z}$ as
\[
{\bm S}(x) = \sum_{j=1}^d x(j).
\]
\begin{lemma}\label{l:beforesigmadd1}
\begin{align}
\bar{X}_k(l) &\ge X_k(l), l=2,3,...,j+1, \label{e:XbarX1}\\
\bar{X}_k(l) &= X_k(l), l= j+2,j+3,...,d, \label{e:XbarX2}\\
{\bm S}( X_k) &= {\bm S}(\bar{X}_k) \label{e:XbarX3}
\end{align}
for $k \le \sigma_{j,j+1}$, $j \in \{0,1,2,3,...,d-1\}$;
\begin{equation}\label{e:XbarX4}
{\bm S}( X_k) \ge {\bm S}(\bar{X}_k) 
\end{equation}
for $k > \sigma_{d-1,d}.$
\end{lemma}
Note that \eqref{e:XbarX3} holds for all $j$, i.e., 
${\bm S}( X_k) = {\bm S}(\bar{X}_k)$, $k \le \sigma_{d-1,d}.$
\begin{proof}
The processes $X$ and $\bar{X}$ have the same dynamics except on $\partial_1$ where only $X$ is constrained and by assumption \eqref{e:sameinitcond}
 they start from the same point.
Therefore, until they hit $\partial_1$ they move together, i.e.,
\[
X_k = \bar{X}_k
\]
for $k \le \sigma_{0,1}.$ These prove \eqref{e:XbarX1}, \eqref{e:XbarX2}
and \eqref{e:XbarX3} for $j=0.$
For $j \ge 1$ we will use induction. Assume 
\eqref{e:XbarX1}, \eqref{e:XbarX2} and \eqref{e:XbarX3} 
hold for $j=j_0 < d-1$;
let us prove that they will also hold for $j=j_0 + 1.$ 
We will do this by another
induction on $k$,
$  \sigma_{j_0, j_0+1} \le k \le \sigma_{j_0 + 1, j_0 + 2}.$ 
Note that there is a nested induction here, one induction on $j$ another on
$k$- we will refer to the induction on $j$ as the outer induction
and to the one on $k$ as the inner induction.
For $k = \sigma_{j_0,j_0 + 1}$ the statements hold by the 
outer induction hypothesis.
Now assume that \eqref{e:XbarX1}, \eqref{e:XbarX2} and
\eqref{e:XbarX3}, $j=j_0 +1$, 
hold for $k \le k_0 $ for some $\sigma_{j_0,j_0+1} \le  k_0 < \sigma_{j_0+1,j_0+2}.$
We want to show that they must also hold for $k=k_0 + 1.$ We argue based on the possible positions of $X$ and $\bar{X}$
at time $k_0$:
\begin{enumerate}
\item $X_{k_0} \in {\mathbb Z}_+^d - \bigcup_{j=1}^d \partial_j$:
this and the inner induction hypothesis imply $\bar{X}_{k_0}(l) > 0$ for $l=2,3,...,d$;
furthermore $\bar{X}$ is not constrained on $\partial_1$. These imply
\[
X_{k+1} = X_k + I_{k+1}, \bar{X}_{k+1} = \bar{X}_k + I_{k+1},
\]
i.e., both $X$ and $\bar{X}$ change by the same increment. Therefore, 
all of the relations \eqref{e:XbarX1}, \eqref{e:XbarX2} and
\eqref{e:XbarX3}
are preserved from time $k=k_0$ to $k=k_0 + 1.$

\item $X_{k_0}$ can also be on the boundary of ${\mathbb Z}_+^d$; recall that $\sigma_{j_0 + 1,j_0 + 2}$ is the first
time $X$ hits $\partial_{j_0 + 2}$ after time $\sigma_{j_0, j_0+1}$. Therefore, $X_{k_0} \notin \partial_{j_0 + 2}$, since
$\sigma_{j_0, j_0 +1} \le k_0 < \sigma_{j_0 + 1, j_0 + 2}$.
Then if $X_{k_0}$ is on the boundary of ${\mathbb Z}_+^d$ it must be on one of the following:
\[
X_{k_0} \in  
\partial_{\mathcal M} \doteq 
\left(\bigcap_{m \in {\mathcal M}} \partial_m \right) \bigcap\left( \bigcap_{m \in {\mathcal M}^c} \partial_m^c\right) 
\]
for some
${\mathcal M} \subset\{1,2,3,...,j_0 + 1, j_0+3,...,d\}$: 
\begin{enumerate}
\item 
if $I_{k_0 + 1} = e_1$, or $I_{k_0 + 1} = -e_m + e_{m+1}$ for some $m \in {\mathcal M}^c$: 
the increment $e_1$ is not constrained for $X$ and $\bar{X}$ regardless of their position. For  the case
$I_{k_0 +1} = -e_m + e_{m+1}$:  $X_{k_0} \in \partial_m^c$ means $X_{k_0}(m) > 0$. This 
and the inner induction hypothesis (\eqref{e:XbarX1} and \eqref{e:XbarX2}) imply $\bar{X}_{k_0}(m) > 0$ if $m > 1$; furthermore
$\bar{X}$ is not constrained on $\partial_1$. These imply
\[
X_{k_0+1} = X_{k_0} + I_{k_0+1}, \bar{X}_{k_0+1} = \bar{X}_{k_0} + I_{k_0+1}.
\]
Once again this implies that the relations \eqref{e:XbarX1} and \eqref{e:XbarX2} are preserved
from time $k=k_0$ to $k= k_0 + 1.$ 
\item If $I_{k_0 + 1} = -e_m + e_{m+1}$ for $m \ge  (j_0 +1)+2$, $m \in {\mathcal M}$: 
$X_{k_0} \in \partial_{\mathcal M}$ implies $X_{k_0}(m) = 0.$
By the inner induction hypothesis $\bar{X}_{k_0}(m) = X_{k_0}(m)$ for $m \ge (j_0 + 1) + 2$.
Therefore, $\bar{X}_{k_0}(m) = 0$ as well. These imply that the increment
$-e_m + e_{m+1}$ is constrained both for $X$ and $\bar{X}$:
\begin{equation}\label{e:bothconstrained}
X_{k_0+1} = X_{k_0},  \bar{X}_{k_0+1} = \bar{X}_{k_0},
\end{equation}
and the relations \eqref{e:XbarX1}, \eqref{e:XbarX2} and \eqref{e:XbarX3} are trivially preserved
from time $k=k_0$ to $k= k_0 + 1.$ 
\item  If $I_{k_0 + 1} = -e_m + e_{m+1}$, $2 \le m  \le j_0 + 1$, $m \in {\mathcal M}$: we know by the induction hypothesis that
$\bar{X}_{k_0}(m) \ge X_{k_0}(m)$. If $\bar{X}_{k_0}(m) = X_{k_0}(m)$ then the increment $-e_m + e_{m+1}$ is constrained both
for $X$ and $\bar{X}$, \eqref{e:bothconstrained} holds and 
the relations \eqref{e:XbarX1}, \eqref{e:XbarX2} and \eqref{e:XbarX3}
are trivially preserved
from time $k=k_0$ to $k= k_0 + 1.$  If $\bar{X}_{k_0}(m) > X_{k_0}(m)$ then the increment $-e_m  + e_{m+1}$ is unconstrained
for $\bar{X}$ while it is constrained for $X$:
\[
X_{k_0+1} = X_{k_0},  \bar{X}_{k_0+1} = \bar{X}_{k_0} - e_m + e_{m+1}.
\]
The linearity of ${\bm S}$ and ${\bm S}(-e_m + e_{m+1}) = 0$ imply
that \eqref{e:XbarX3} is preserved at time $k_0 + 1.$ 
All of the components $\bar{X}(l)$, $l \neq m, m+1$ remain unchanged from $k_0$ to $k_0 + 1$.
Therefore, 
the relations \eqref{e:XbarX1} and \eqref{e:XbarX2} are trivially preserved for these components; in particular, this shows that
\eqref{e:XbarX2} holds at time $k_0 + 1$ with $j=j_0 + 1$
because $m, m+1 \le (j_0 + 1) + 2.$
To complete the proof it suffices to show that \eqref{e:XbarX1}
holds for $j=j_0 + 1$ and $k= k_0 +1$ for components $l=m$ and $l=m+1$.
For $l = m+1$, 
\[
\bar{X}_{k_0+1}(m+1) = 
\bar{X}_{k_0}(m+1) + 1 \ge X_{k_0}(m+1) = X_{k_0 + 1}(m+1).
\]
For $l= m$: recall that we are treating the case
 $\bar{X}_{k_0}(m) > X_{k_0}(m)$, i.e., 
 $\bar{X}_{k_0}(m) \ge X_{k_0}(m) + 1$. Therefore:
\[
\bar{X}_{k_0+1}(m+1) = \bar{X}_{k_0}(m+1) - 1 \ge X_{k_0}(m+1) = X_{k_0 + 1}(m+1);
\]
these prove that \eqref{e:XbarX1} holds at time $k_0 + 1$ with $j=j_0 + 1.$
\item Finally, it may happen that $1 \in {\mathcal M}$ and
$I_{k_0 + 1} = -e_1 + e_2$. In this case, $X_{k_0} \in \partial_1$
and therefore the increment $I_{k_0 + 1}$ is canceled by the constraining
map $\pi$ for $X$; $\bar{X}$ is unconstrained on $\partial_1$, therefore,
the increment $I_{k_0+1}$ is not constrained for $\bar{X}$. Therefore,
\[
\bar{X}_{k_0 + 1}(l) = \bar{X}_{k_0}(l), 
{X}_{k_0 + 1}(l) = X_{k_0}(l), l =3,4,...,d,
\]
and
\[
\bar{X}_{k_0 + 1}(2) = \bar{X}_{k_0}(2) + 1,
{X}_{k_0 + 1}(2) = X_{k_0}(2).
\]
These imply that the relation \eqref{e:XbarX1} and \eqref{e:XbarX2} for
$j = j_0 +1$
are preserved from time $k_0$ to $k_0 + 1.$
The preservation of \eqref{e:XbarX3} follows from the linearity of
${\bm S}$ and ${\bm S}(-e_1 + e_2) =1$ as in the last part.
\end{enumerate}
\end{enumerate}
This case by case analysis completes the inner induction step and hence
the outer induction step.

$X_k \in \partial_d$ for $k=\sigma_{d-1,d}$. 
If
$I_{k+1} = -e_d$ and $\bar{X}_k \notin \partial_d$ we have:
\[
X_{k+1} = X_k, \bar{X}_{k+1} = \bar{X_k} - e_d;
\]
an application of ${\bm S}$ to both sides of the above equations
and \eqref{e:XbarX3} imply ${\bm S}(X_{k+1}) = {\bm S}(\bar{X}_{k+1}) +1$;
thus 
${\bm S}(X_{k+1}) > {\bm S}(\bar{X}_{k+1})$ can happen
after time $\sigma_{d-1,d}$.
A case by case analysis parallel to the one given above shows that 
\eqref{e:XbarX4} is preserved at all times after $\sigma_{d-1,d}.$
\end{proof}

The previous lemma implies
\begin{lemma}\label{l:comparetau}
The stopping times $\tau_n$, $\bar\tau_n$ and $\sigma_{d-1,d}$ satisfy:
\begin{enumerate}
\item
for any $n \ge 0$, $\sigma_{d-1,d} \ge \tau_n $ if and only if
$\sigma_{d-1,d} \ge \bar{\tau}_n.$
\item
\begin{equation}\label{e:equaltau}
\tau_n = \bar{\tau}_n 
\end{equation}
over the event $\{ \sigma_{d,d+1} \ge \tau_n\}=\{ \sigma_{d,d+1} \ge \bar\tau_n\}.$
\item  \begin{equation}\label{e:comparetaun1}
\tau_n \ge \bar{\tau}_n 
\end{equation}
if $n < {\bm S}(x)$
and\begin{equation}\label{e:comparetaun2}
\tau_n \le \bar\tau_n 
\end{equation}
if $ n > {\bm S}(x).$
\end{enumerate}
\end{lemma}
\begin{proof}
By definition
$\tau_n \le \sigma_{d-1,d}$ if and only if
\[
{\bm S}(X_k) =n,
\]
for some $k \le \sigma_{d-1,d}$ and
$\bar\tau_n \le \sigma_{d-1,d}$ if and only if
\[
{\bm S}(\bar{X}_k) =n,
\]
for some $k \le \sigma_{d-1,d}$.
By the previous lemma ${\bm S}(X_k) = {\bm S}({\bar X}_k)$ for 
$k \le \sigma_{d-1,d}.$ These imply the first two parts of the current lemma.
Similarly,
\[
\tau_n = \inf \{k \ge 0: {\bm S}(X_k) = n \},~
\bar\tau_n = \inf \{k \ge 0: {\bm S}(\bar{X}_k) = n \};
\]
${\bm S}(\bar{X}_k) = {\bm S}(X_k)$ for $k \le \sigma_{d-1,d}$
by Lemma \ref{l:beforesigmadd1}, \eqref{e:XbarX3}. Therefore,
for $\tau_n \le \sigma_{d,d+1}$
\[
\tau_n = \inf \{\sigma_{d-1,d} \ge k \ge 0: {\bm S}(X_k) = n \}
 = \inf \{\sigma_{d-1,d} \ge k \ge 0: {\bm S}(\bar{X}_k) = n \} = \bar{\tau}_n,
\]
i.e, \eqref{e:equaltau} holds.

The relations \eqref{e:XbarX3} and \eqref{e:XbarX4} imply that
\begin{equation}\label{e:alwaysless}
{\bm S}(\bar{X}_k) \le {\bm S}(X_k).
\end{equation}
for all $k \ge 0.$
We will argue the case when $n < {\bm S}(x)$, the opposite case is argued
similarly. By definition, ${\bm S}(X_{\tau_n}) = n.$ This and \eqref{e:alwaysless} imply ${\bm S}(\bar{X}_{\tau_n}) \le n.$ The process ${\bm S}(\bm \bar X)$ 
jumps by increments of $-1$ (happens when $\bar X$ jumps by $-e_d$)
 and $1$ (happens when $\bar X$ jumps by $e_1$). It follows that $\bar{X}$
must take all of the values $n,n+1,n+2,...,{\bm S}(x)$ in the time interval
$k \in \{0,1,2,...., \tau_n\}.$ This implies $\bar \tau_n \le \tau_n.$

\end{proof}

We now express the difference between the events
$\{\tau_n \le \tau_0\}$ and $\{\tau < \infty \} = \{\bar\tau_n < \infty \}$
in terms of the stopping times $\sigma_{d-1,d}$ and $\bar\tau_0.$
For two events $A$ and $B$ let $\Delta$ denote their symmetric difference:
$A \Delta B \doteq ( A - B) \cup (B - A).$
\begin{lemma}\label{l:boundondifference}
For $X_0 = x$,  $0 < {\bm S}(x) < n$
\begin{equation}\label{e:boundondifference}
\{\tau_n < \tau_0\}  \Delta \{ \bar\tau_n < \infty\}
\subset
\{ \tau_n < \tau_0,  \tau_n >\sigma_{d-1,d} \}
\cup \{ \bar\tau_0  < \bar\tau_n < \infty\}
\end{equation}
holds.
\end{lemma}
\begin{proof}
Break down $\{\tau_n < \tau_0\}$ and
$\{\bar\tau_n < \infty\}$ into two as
\begin{align}\label{e:decompose0}
\{ \tau_n < \tau_0 \} &= \{ \tau_n < \tau_0,  \tau_n \le \sigma_{d-1,d} \} \cup \{ \tau_n < \tau_0,  \tau_n > \sigma_{d-1,d} \},\\
\{ \bar\tau_n < \infty \} &= \{ \bar\tau_n < \infty,  \bar\tau_n \le \sigma_{d-1,d} \} \cup \{ \bar\tau_n < \infty,  \bar\tau_n > \sigma_{d-1,d} \}.
\notag
\end{align}
That $\tau_n = \bar\tau_n$ for  $\tau_n \le \sigma_{d-1,d}$ and
$\sigma_{d-1,d} \le \tau_n$ if and only if 
$\sigma_{d-1,d} \le \bar\tau_n$ imply
\begin{equation}\label{e:decompose1}
\{ \tau_n < \tau_0,  \tau_n \le \sigma_{d-1,d} \}
\subset 
\{ \bar\tau_n < \infty ,  \bar\tau_n \le \sigma_{d-1,d} \}.
\end{equation}
On the other hand,
\begin{equation}\label{e:decompose2}
\{ \bar\tau_n < \infty ,  \bar\tau_n \le \sigma_{d-1,d} \} = 
\{ \bar\tau_n < \infty ,  \bar\tau_n \le \sigma_{d-1,d}, \bar\tau_0  < \bar\tau_n \}  \cup
\{ \bar\tau_n < \infty ,  \bar\tau_n \le \sigma_{d-1,d}, \bar\tau_0 > \bar\tau_n \};
\end{equation}
$\tau_0 \ge \bar\tau_0$ 
by \eqref{e:comparetaun2}  and 
$\tau_n = \bar\tau_n$
for $\bar\tau_n \le \sigma_{d-1,d}$ by \eqref{e:equaltau};therefore
\[
\{ \bar\tau_n < \infty ,  \bar\tau_n \le \sigma_{d-1,d}, \bar\tau_0 > \bar\tau_n \} \subset
\{ \tau_n \le \sigma_{d-1,d}, \tau_0 > \tau_n \}.
\]
The last line, \eqref{e:decompose0}, \eqref{e:decompose1} and \eqref{e:decompose2} 
imply
\begin{align}\label{e:almostthere}
&\{\tau_n < \tau_0\}  \Delta \{ \bar\tau_n < \infty\}\\
&~~~~\subset \{ \tau_n < \tau_0,  \tau_n >\sigma_{d-1,d} \}
\cup \{ \bar\tau_n < \infty,  \bar\tau_n > \sigma_{d-1,d} \}
\cup \{ \bar\tau_n < \infty ,  \bar\tau_n \le \sigma_{d-1,d}, \bar\tau_0 \le \bar\tau_n \}.\notag
\end{align}
Next we decompose  $\{ \bar\tau_n < \infty,  \bar\tau_n > \sigma_{d-1,d} \}$
into two:
\[
\{ \bar\tau_n < \infty,  \bar\tau_n > \sigma_{d-1,d} \}
=
\{ \bar\tau_n < \infty,  \bar\tau_n > \sigma_{d-1,d} , \bar\tau_0 < \bar\tau_n \} \cup
\{ \bar\tau_n < \infty,  \bar\tau_n > \sigma_{d-1,d} , \bar\tau_0 >
\bar\tau_n \}.
\]
The assumption $ 0 < {\bm S}(x) < n$, \eqref{e:comparetaun1} and
\eqref{e:comparetaun2} imply 
$\tau_n \le \bar\tau_n$ and $\tau_0 \ge \bar\tau_0$; furthermore by
the first part of Lemma \ref{l:comparetau}
$\tau_n > \sigma_{d-1,d}$ if and only if $\bar\tau_n  > \sigma_{d-1,d}$;
these imply
\[
\{ \bar\tau_n < \infty,  \bar\tau_n > \sigma_{d-1,d} , \bar\tau_0 > \bar\tau_n\}
\subset
\{ \tau_n > \sigma_{d-1,d} , \tau_0 >\tau_n\}.
\]
The last two displays give:
\[
\{ \bar\tau_n < \infty,  \bar\tau_n > \sigma_{d-1,d} \}
\subset
\{ \bar\tau_n < \infty,  \bar\tau_n > \sigma_{d-1,d} , \bar\tau_0 < \bar\tau_n \} \cup
\{ \tau_n > \sigma_{d-1,d} , \tau_0 >\tau_n\}.
\]
This and \eqref{e:almostthere} imply \eqref{e:boundondifference}.
\end{proof}
By this lemma the numerator of the relative error \eqref{e:relerr}
is bounded by
\begin{align*}
| {\mathbb P}_{x_n}(\tau_n < \tau_0 ) - {\mathbb P}_{T_n(x_n)}(\tau < \infty) | &\le 
| {\mathbb P}_{x_n}( \{\tau_n < \tau_0 \} \Delta \{\bar\tau_n < \infty\}) \\
&\le 
{\mathbb P}_{x_n}( \tau_n < \tau_0,  \tau_n >\sigma_{d-1,d} )
+{\mathbb P}_{x_n}(\bar\tau_0  < \bar\tau_n < \infty).
\end{align*}
The next two subsections derives upperbounds on the last two probabilities.
The one following them finds a lower bound on the denominator 
${\mathbb P}_{x_n}(\tau_n < \tau_0)$ of the relative error. The last
subsection combines these to give a proof of Theorem \ref{t:convergence}.

\subsection{Upperbound on ${\mathbb P}_x(\bar\tau_n < \infty)$}\label{ss:tauy}
Recall that 
${\mathbb P}_x(\bar\tau_n < \infty = {\mathbb P}_{T_n(x)}( \tau < \infty)$.
The function
$y\mapsto {\mathbb P}_{y}( \tau < \infty)$
is $Y$-harmonic.
In Sections \ref{s:dg2} and \ref{s:hstq} we will compute this
$Y$-harmonic function exactly and see that
$y\mapsto {\mathbb P}_{y}( \tau < \infty)$
has a rather intricate structure.
It turns out to
be possible to derive the upperbounds we need
using much simpler $Y$-superharmonic
functions and in this subsection that is what we will do.
We will derive an upperbound on the probability 
${\mathbb P}_{x_n}(\bar\tau_n < \infty)={\mathbb P}_{T_n(x)}(\tau <  \infty)$;
the  bound we seek on ${\mathbb P}_{x_n}(\bar\tau_0 < \bar\tau_n < \infty)$
will follow from the bound on ${\mathbb P}_{y}(\tau <  \infty)$
by the Markov property of $Y$.

The $n$ variable plays no role here and therefore we will derive the bound and
the superharmonic functions in terms of the $Y$ process- the bound for the 
$\bar{X}$-process will follow by the change of variable $x = T_n(y).$

A real valued function $h$ is said to be $Y$-superharmonic on a
set $A \subset D_Y$ if
\[
{\mathbb E}_y[h(Y_1)] \le h(y), y \in A.
\]
We say $h$ is $Y$-superharmonic if $A= D_Y$.

Define
\[
h_{k,r} \doteq r^{y(1) - \sum_{j=2}^k y(j)}, k  \in \{1,2,3,...,d\}, r > 0.
\]
\begin{proposition}\label{p:hkr}
The function $h_{k,r}$ satisfies
\begin{equation}\label{e:howsuperharmonic}
{\mathbb E}_y[h_{k,r}(Y_1)] -h_{k,r}(y) =
\begin{cases}
h_{k,r}(y) \left(
\lambda \left(\frac{1}{r}-1\right) + \mu_1 (r-1)\right), &\text{if } k=1,\\
h_{k,r}(y) \left(
\lambda \left(\frac{1}{r}-1\right) + \mu_k (r-1) {\bm 1}_{\partial_k^c}(y)\right),
&\text{if } k \in \{2,3,...,d\}.
\end{cases}
\end{equation}
In particular, for $r \in (\rho,1)$,  $h_{1,r}$ is $Y$-superharmonic and for 
$k \in \{2,3,4,...,d\}$ $h_{k,r}$ is $Y$-superharmonic on $D_Y - \partial_k.$
\end{proposition}
In the proof we will use the following basic fact:
\begin{lemma}
For $r \in (\rho,1)$ and any $k \in \{1,2,3,4,...,d\}$
\begin{equation}\label{e:strictlynegative}
\lambda (\frac{1}{r}-1) + \mu_k (r-1) < 0.
\end{equation}
\end{lemma}
\begin{proof}
The function
$r\mapsto \lambda (\frac{1}{r}-1) + \mu_k (r-1)$
is convex for $r \in (0,\infty)$ and equals $0$ for $r=\lambda/\mu_k < \rho$
and $r=1$. It follows that it is strictly below $0$ on the interval
$(\rho,1).$ 
\end{proof}
\begin{proof}[Proof of Proposition \ref{p:hkr}]
For \eqref{e:howsuperharmonic}
The distribution of $Y_1$
and the definition of $h_{k,r}$ imply
\begin{equation}\label{e:EY1}
{\mathbb E}_y[h_{k,r}(Y_1)] =
h_{k,r}(y) \left(
\lambda \frac{1}{r} + \mu_k r{\bm 1}_{\partial_k^c}(y)
+ \mu_k {\bm 1}_{\partial_k}(y) + \sum_{j \neq k} \mu_j \right);
\end{equation}
subtracting $ h_{k,r}(y) = h_{k,r}(y) \left( \lambda_1 + \sum_{j=1}^d \mu_j\right)$ from the last expression gives \eqref{e:howsuperharmonic}.
For $k=1$, $Y_1$ is not constrained on $\partial_1$, therefore
\eqref{e:EY1} in that case reduces to
\[
{\mathbb E}_y[h_{k,r}(Y_1)] =
h_{k,r}(y) \left(
\lambda \frac{1}{\rho} + \mu_1 \rho + \sum_{j \neq 1} \mu_j \right).
\]
The rest of the argument remains the same for $k=1.$
The inequality \eqref{e:strictlynegative} and \eqref{e:howsuperharmonic} imply
\begin{align*}
{\mathbb E}_y[h_{k,r}(Y_1)] -h_{k,r}(y) &< 0 , y \in D_Y, \text{ if } k=1,\\
{\mathbb E}_y[h_{k,r}(Y_1)] -h_{k,r}(y) &< 0 ,  y \in D_Y -\partial_k\text{ if } k\in \{2,3,...,d\}.
\end{align*}
This proves the $Y$-superharmonicity of $h_{k,r}$ (on $D_Y$ for $k=1$
and on $D_Y-\partial_k$ for $k> 1$).
\end{proof}
Linear combinations of $h_{k,r}$ give
further $Y$-superharmonic functions:
define the constants
\begin{equation}\label{d:gamma}
\gamma_1 \doteq 1, 
\gamma_k \doteq \frac{1}{d}\frac{\min_{j < k}\gamma_j(\lambda(1 -1/r) + \mu_{j}(1-r))}{\lambda(1/r - 1) }, k=2,3,...,d.
\end{equation}
By \eqref{e:strictlynegative}, $\gamma_k > 0$ for  $r \in (\rho,1)$.
Now define
\[
h_{2,k,r} \doteq \sum_{j=1}^k \gamma_j h_{j,r}.
\]
\begin{proposition}\label{p:Ysuperharmonicfuns}
For any $k=1,2,3,...d$ and $r \in (\rho,1)$
the function $h_{2,k,r}$ is $Y$-superharmonic
\end{proposition}
\begin{proof}
We assume throughout that $r \in (\rho,1).$
The proof is by induction. By definition $h_{2,k,r} = h_{k,r}$ for $k=1$
and we know that $h_{1,r}$ is $Y$-superharmonic by the previous proposition.
Now assume that $h_{2,k,r}$ is $Y$-superharmonic for some $k< d$; we will prove
that $h_{2,k+1,r}$ must also be $Y$-superharmonic. The function
$h_{k+1,r}$ is $Y$-superharmonic on $D_Y - \partial_{k+1}$ by the
previous proposition; the function $h_{2,k,r}$ is $Y$-superharmonic
by the induction hypothesis. These
and  $\gamma_{k+1} > 0$ imply that $h_{2,k+1,r}$ is $Y$-superharmonic
on $D_Y - \partial_{k+1}.$ Therefore, it suffices to prove that
$h_{2,k+1,r}$ is $Y$-superharmonic on $\partial_{k+1}.$
Choose any $y\in \partial_{k+1}$ and let
\[
k_0 = \max\{j \le k+1: y(j) > 0 \} \vee 1,
\]
where, by convention the $\max$ of the empty set is $-\infty.$
By the induction hypothesis $h_{2,k_0-1,r}$ is $Y$-superharmonic on
$\partial_k$.
Therefore it suffices to prove that
\[
h_{2,k_0,k+1,r} \doteq \sum_{j=k_0}^{k+1} \gamma_j h_{j,r}
\]
satisfies the $Y$-superharmonicity condition for the chosen $y$.
The definition of $k_0$ implies $y(j) = 0$ for $k_0 < j \le k+1.$
This and the definition of $h_{k,r}$ imply
$h_{j,r}(y) = h_{k_0,r}(y)$ for all $k_0 \le j \le k+1$.
Then
\begin{align}
{\mathbb E}_y&[h_{2,k_0,k+1,r}(Y_1)] -h_{2,k_0,k+1,r}(y) \notag\\
&=
h_{k_0,r}(y) \left( \lambda \left(\frac{1}{r}-1\right) \sum_{j=k_0+1}^{k+1} \gamma_j
+ \gamma_{k_0} \left( \lambda\left(\frac{1}{r}-1\right) + \mu_{k_0}(r-1)\right)
\right).\label{e:upperboundk0}
\end{align}
The definition \eqref{d:gamma} and $j > k_0$ implies
\[
\gamma_j \le \frac{1}{d} 
\frac{1}{\lambda\left(\frac{1}{r}-1\right) }
\gamma_{k_0} \left( \lambda\left(1-\frac{1}{r}\right) +\mu_{k_0}(1-r)\right).
\]
Substituting this in \eqref{e:upperboundk0} gives
\begin{align*}
{\mathbb E}_y&[h_{2,k_0,k+1,r}(Y_1)] -h_{2,k_0,k+1,r}(y) \\
&\le
h_{k_0,r}(y) 
\gamma_{k_0}\frac{d-(k+1)+k_0}{d} 
\left( \lambda\left(\frac{1}{r}-1\right) +  \mu_{k_0}(r-1)\right) 
 < 0,
\end{align*}
which completes the induction step.
\end{proof}
Applying the $Y$-harmonic function $h_{2,d,r}$ to the process $Y$
we obtain the supermartingale
$h_{2,d,r}(Y_k)$, which gives us the bound we seek
on ${\mathbb P}_y(\tau < \infty)$:
\begin{proposition}
For $y \in D_Y$ and $ r \in (\rho, 1)$
\begin{equation}\label{e:boundonPtaufinite}
{\mathbb P}_y(\tau < \infty) \le\frac{1}{\gamma_d} h_{2,d,r}(y).
\end{equation}
\end{proposition}
\begin{proof}
Let $m$ be a positive integer.
The optional sampling theorem applied to the supermartingale
$k \mapsto h_{2,d,r}(Y_k)$ at the stopping time $m \wedge \tau$ gives
\begin{align*}
h_{2,d,r}(y) &\ge {\mathbb E}_{y}[ h_{2,d,r}(Y_{\tau \wedge m})]\\
&= {\mathbb E}_{y}[ h_{2,d,r}(Y_{\tau}){\bm 1}_{\{\tau \le m \}}]
 +{\mathbb E}_{y}[ h_{2,d,r}(Y_m){\bm 1}_{\{\tau >m \}}].
\end{align*}
By definition $h_{2,d,r} \ge 0$ and 
$h_{2,d,r}/\gamma_d \ge h_{d,r} = 1$ on $\partial B.$ These
and the previous display imply
\[
h_{2,d,r}(y)/\gamma_d \ge {\mathbb P}_{y}(\tau \le m ).
\]
Letting $m \rightarrow \infty$ gives \eqref{e:boundonPtaufinite}.
\end{proof}
By definition
\begin{equation}\label{e:reminder}
{\mathbb P}_x(\bar\tau_n < \infty) = 
{\mathbb P}_{T_n(x)}(\tau < \infty).
\end{equation}
The bound on ${\mathbb P}_x(\bar\tau_0 < \bar\tau_n < \infty)$ now
follows from the previous proposition and the Markov property of
$\bar{X}$:
\begin{proposition}\label{p:boundonbtau0btaun}
For $x \in {\mathbb Z} \times {\mathbb Z}_+^{d-1}$ and $ 0 < {\bm S}(x) < n$
\begin{equation}\label{e:boundonbtau0btaun}
{\mathbb P}_x(\bar\tau_0 < \bar\tau_n < \infty) \le 
\rho^n \frac{1}{\gamma_d} \sum_{j=1}^d \gamma_j.
\end{equation}
\end{proposition}
\begin{proof}
By \eqref{e:reminder}, \eqref{e:boundonPtaufinite} in $x$ coordinates
is
\[
{\mathbb P}_x(\bar\tau_n < \infty) \le 
\frac{1}{\gamma_d}h_{2,d,r}(T_n(x))= 
\frac{1}{\gamma_d}\sum_{j=1}^d \gamma_j h_{d,r}(T_n(x)) = 
\frac{1}{\gamma_d}\sum_{j=1}^d \gamma_j 
r^{n- \sum_{k=1}^{j} x(k)}.
\]
This, $x \in {\mathbb Z} \times {\mathbb Z}_+^{d-1}$ and $0 < r < 1$ imply
\[
{\mathbb P}_x(\bar\tau_n < \infty) \le r^{n}\frac{1}{\gamma_d} \sum_{j=1}^d \gamma_j
\]
for ${\bm S}(x) = 0.$ The previous display is true for any $r \in (\rho,1)$;
by continuity it is also true for $r=\rho.$
The strong Markov property of $\bar{X}$ implies
\[
{\mathbb P}_x(\bar\tau_0 < \bar\tau_n < \infty) =
{\mathbb E}_{x}\left[{\mathbb P}_{\bar{X}_{\bar\tau_0}}( \bar\tau_n < \infty)\right].
\]
This, ${\bm S}(X_{\bar\tau_0}) = 0$ 
and the previous display imply \eqref{e:boundonbtau0btaun}.
\end{proof}
\subsection{Upperbound on ${\mathbb P}_x(\sigma_{d-1,d} < \tau_n < \tau_0)$}\label{ss:sigmad1d}
The goal of this subsection is to prove the following bound:
\begin{proposition}\label{p:upperboundonlong}
For any $\epsilon > 0$ there exists $n_0 > 0$ such that
\[
{\mathbb P}_{x}(\sigma_{d-1,d} < \tau_n < \tau_0) < \rho^{n(1-\epsilon)},
\]
for $n > n_0$ and $x \in A_n.$
\end{proposition}

As in the previous section and as in two dimensions treated in
\cite{sezer2018approximation} we will construct a supermartingale
to upperbound 
${\mathbb P}_x(\sigma_{d-1,d} < \tau_n < \tau_0)$. The event
$\{\sigma_{d-1,d} < \tau_n < \tau_0\}$ consists of at most $d+1$ stages: 
the process $X$ starts on or away from $\partial_1$, then hits $\partial_2$
, then hits $\partial_3$, etc. and finally hits $\partial A_n$ after hitting
$\partial_d$ without ever hitting $0$. Roughly, 
the supermartingale will be constructed
by applying one of the functions $h_{2,k,r}$ to the process $X$ at
each of these stages. The next lemma is used to adjust the definition
so that the defined process remains a supermartingale as $X$ jumps from one
stage to the next.

For $k \in \{2,3,...,d\}$ define
\[
\gamma_{k-1,k} \doteq \frac{\gamma_{k-1}}{\gamma_{k-1} + \gamma_{k}}.
\]
\begin{lemma}\label{l:inequalityoverjumps}
For $k \in \{2,3,...,d\}$  and $r \in (\rho,1)$
\begin{equation}\label{e:lowerboundonratio}
\min_{y \in \partial_{k}} \frac{h_{2,k-1,r}(y)}{h_{2,k,r}(y)} \ge 
\gamma_{k-1,j}
\end{equation}
for $y \in \partial_{k}.$
\end{lemma}
\begin{proof}
By their definition
\[
h_{k-1,r}(y) = h_{k,r}(y)  = r^{y(1)- \sum_{j=2}^{k-1} y(j)}.
\]
for $y \in \partial_{k}.$
This and the definition of $h_{2,k-1,r}, h_{2,k,r}$ imply
\[
\frac{h_{2,k-1,r}(y)}{h_{2,k,r}(y)} 
=
\frac{
r^{y(1)- \sum_{j=1}^{k-1} y(j)} \left( \gamma_{k-1} + \sum_{j=1}^{k -2}
 \gamma_j r^{\sum_{l=2}^j y(l)} \right)}
{r^{y(1)- \sum_{j=1}^{k-1} y(j)} \left( \gamma_{k} +\gamma_{k-1}+ 
\sum_{j=1}^{k -2}
 \gamma_j r^{\sum_{l=2}^j y(l)} \right)}
\]
for $y \in \partial_{k}$;
this and $\gamma_j, r > 0$ imply
\eqref{e:lowerboundonratio}.
\end{proof}
Define
\[
\Gamma_j \doteq \prod_{l=2}^{j} \gamma_{l-1,l},
\]
and
\[
S_k' \doteq 
\Gamma_j h_{2,j,r}(T_n(X_k)) \text{ for }\sigma_{j-1,j} < k \le \sigma_{j,j+1}, 
j=0,1,2,3,...,d,
\]
where, by convention, $\Gamma_{0} = \Gamma_1 = 1$, $h_{2,0,r} = r^n$,
 $\sigma_{-1,0} = -1$ and $\sigma_{d,d+1} = \infty$; in particular,
$S_k' = r^n$ for $k \le \sigma_{0,1}$ and $S_k' = \Gamma_d h_{2,d,r}(T_n(X_k))$
for $k > \sigma_{d-1,d}.$
The supermartingle that we will use to upperbound the probability
${\mathbb P}_{x}(\sigma_{d-1,d} < \tau_n < \tau_0)$
is
\begin{equation}\label{e:defsupS}
S_k \doteq S_k' - 
k\left( \lambda \left(\frac{1}{r} -1\right) \sum_{j=1}^d \gamma_j \right) r^n.
\end{equation}
\begin{proposition}
The process $\{S_k, k=0,1,2,3,...\}$ is a supermartingale.
\end{proposition}
\begin{proof}
The proof is a case by case analysis. We begin by $X_k \notin \partial_1$,
i.e., $X_k(1) > 0.$ There are two subcases two consider:
$k = \sigma_{j-1,j}$ for some $j \ge 1$ and 
$k \neq \sigma_{j-1,j}$ for all $j \ge 1$.
For 
$k \neq \sigma_{j-1,j}$, 
$S_k' = \Gamma_j h_{2,j,r}(T_n(X_k))$ 
$S_{k+1}' = \Gamma_j h_{2,j,r}(T_n(X_{k+1}))$ 
for some $j \in \{0,1,2,3,...,d\}$;
the functions $h_{2,j,r}$ are $Y$-superharmonic by Proposition
\ref{p:Ysuperharmonicfuns} and therefore 
$h_{2,j,r}(T_n(\cdot))$ are $X$-superharmonic on $\partial_1^c.$
It follows from these that
\begin{equation}\label{e:interiorinequalityskprime}
{\mathbb E}_x[S_{k+1}' | {\mathscr F}_k] \le S_k'
\end{equation}
over the event 
$\{ X_k \notin \partial_1\} \cap \{k \neq \sigma_{j-1,j}, j=1,2,3,...,d\}\}.$
If $k = \sigma_{j-1,j}$ for some $j \in \{2,3,4,...,d\}$ we have
$S_k' = \Gamma_{j-1}h_{2,j-1,r}(T_n(X_k))$
and $S_{k+1}' = \Gamma_j h_{2,j,r}(T_n(X_{k+1})$; $h_{2,j,r}$  is $Y$-superharmonic
and therefore $h_{2,j,r}(T_n(\cdot))$is $X$-superharmonic on $\partial_1^c$.
These imply
\begin{equation}\label{e:step1h2jr}
\Gamma_j h_{2,j,r}(T_n(X_{k}) \ge
{\mathbb E}[\Gamma_j h_{2,j,r}(T_n(X_{k+1}) |{\mathscr F}_k] =
{\mathbb E}[S_{k+1}' |{\mathscr F}_k] 
\end{equation}
over the event $\{k = \sigma_{j-1,j}\} \cap \{ X_k \notin \partial_1\}.$
That $X_{k} \in \partial_j$ for $k=\sigma_{j-1,j}$
and Lemma \ref{l:inequalityoverjumps}
imply
\[
S_{k}'=\Gamma_{j-1} h_{2,j-1,r}(T_n(X_{k}))\ 
\ge \Gamma_j h_{2,j,r}(T_n(X_k)).
\]
This and \eqref{e:step1h2jr} imply
\[
S_k' \ge {\mathbb E}[S_{k+1}' |{\mathscr F}_k] 
\]
over the event $\{k = \sigma_{j-1,j}\} \cap \{ X_k \notin \partial_1.\}.$
This and \eqref{e:interiorinequalityskprime} imply
$S_k' \ge {\mathbb E}[S_{k+1}' |{\mathscr F}_k]$;
subtracting 
$k\left( \frac{\lambda}{r} + \sum_{j=1}^d \mu_j \right) r^n$
from the left and 
$(k+1)\left( \frac{\lambda}{r} + \sum_{j=1}^d \mu_j \right) r^n$
from the right gives
\[
S_k \ge {\mathbb E}[S_{k+1} |{\mathscr F}_k] 
\]
over the event $\{X_k \notin \partial_1\}$. 

For $x \in {\mathbb Z}_+^d$, define
\[
L^*(x) \doteq \{ l \in \{2,3,4,...d\}, x(l) \neq 0 \};
\]
and 
\[
d^*(x) \doteq |L^*(x)|;
\]
$l_1(x) < l_2(x) < \cdots < l_{d^*}(x)$ are the members of $L^*$; two
conventions 1)
$l_{d^*+1}(x) = d+1$ and 2) $l_1 = d+1$
 if $d^* = 0$, i.e., if $L^*(x) = \emptyset.$
For $X_k \in \partial_1$, 
there are two cases to consider: 1) $k = \sigma_{j-1,j}$
for some $j$ and 2)  $k \neq \sigma_{j-1,j}$ for all $j$.
For the latter case 
\begin{equation}\label{e:skpp1}
S_k' = \Gamma_j h_{2,j,r}(T_n(X_k)),
S_{k+1}' = \Gamma_j h_{2,j,r}(T_n(X_{k+1})),
\end{equation}
for some $j$.
For ease of notation, let us abbreviate $l_m(X_{k})$ to $l_m$, 
and $d^*(X_{k})$ to $d^*$.
Decompose
$h_{2,j,r}(T_n(x))$ as
\begin{equation}\label{e:decomposeh2jr}
h_{2,j,r}(T_n(x)) = 
\sum_{l=1}^{(l_1-1) \wedge j}\gamma_l h_{l,r}(T_n(x)) 
+ \sum_{m=1}^{d^*} \sum_{l=l_m}^{(l_{m+1} -1)\wedge j} \gamma_l h_{l,r}(T_n(x)),
\end{equation}
where we use the conventions set above.
Let us begin by considering any of the inner sums in the
second sum in the last display. By the definition
of $l_m$ and $l_{m+1}$, $X_k(l_m) > 0$ and
$X_k(l) = 0$ for $l_m < l < l_{m+1}$. This implies
$h_{l,r}(T_n(X_k))  = h_{l_m,r}(T_n(X_k))$ 
for $l_m < l < l_{m+1}$. These, $l_m > 1$, Proposition \ref{p:hkr},
the definition \eqref{d:gamma} of $\gamma_l$ and the dynamics of $X$ imply
\begin{align*}
&{\mathbb E}
\left
[\sum_{l=l_m}^{(l_{m+1} -1)\wedge j} \gamma_l h_{l,r}(T_n(X_{k+1}))
| {\mathscr F}_k \right] -
\sum_{l=l_m}^{(l_{m+1} -1)\wedge j} \gamma_l h_{l,r}(T_n(X_k))\\
&~~~~=
h_{l,r}(T_n(X_k))
\left(
\lambda \left( \frac{1}{r} - 1 \right)
\sum_{l=l_m+1}^{(l_{m+1} -1)\wedge j} \gamma_l 
+ \gamma_{l_m} \left( \lambda\left(\frac{1}{r} -1\right) 
+ \mu_{l_m}(r-1) \right)\right)
\le 0.
\end{align*}
Summing the last inequality over $m$ gives
\begin{equation}\label{e:therestSkp}
{\mathbb E}
\left
[\sum_{m=1}^{d^*}\sum_{l=l_m}^{(l_{m+1} -1)\wedge j} \gamma_l h_{l,r}(T_n(X_{k+1}))
| {\mathscr F}_k \right] -
\sum_{m=1}^{d^*}\sum_{l=l_m}^{(l_{m+1} -1)\wedge j} \gamma_l h_{l,r}(T_n(X_k))\\
\le 0.
\end{equation}
Similarly, the definition of $l_1$ implies, $X_k(l) = 0$ for $l< l_1$,
this and $h_{l,r}(T_n(x)) = r^{n-\sum_{m=1}^l x(m)}$
imply $h_{l,r}(T_n(X_k)) = r^n$ for $l < l_1$
over the event $\{X_k \in \partial_1\}$.
These 
and the dynamics of $X$ imply
\begin{align*}
&{\mathbb E}
\left
[\sum_{l=1}^{(l_1 -1)\wedge j} \gamma_l h_{l,r}(T_n(X_{k+1}))
| {\mathscr F}_k \right] -
\sum_{l=1}^{(l_{1} -1)\wedge j} \gamma_l h_{l,r}(T_n(X_k))\\
&~~~~=
h_{1,r}(T_n(X_k))
\left(
\lambda \left( \frac{1}{r} - 1 \right)
\sum_{l=1}^{(l_{1} -1)\wedge j} \gamma_l 
\right) 
= 
r^n 
\left(
\lambda \left( \frac{1}{r} - 1 \right)
\sum_{l=1}^{(l_{1} -1)\wedge j} \gamma_l 
\right) 
\ge 0
\end{align*}
over the event $\{ X_k \in \partial_1\}.$ Putting together the last
display, \eqref{e:therestSkp}, \eqref{e:decomposeh2jr}
and $0 < \Gamma_j \le 1$
give
\begin{equation}\label{e:inforp1}
{\mathbb E}[ \Gamma_{j} h_{2,j,r}(T_n(X_{k+1}))| {\mathscr F}_k]
\le 
\Gamma_{j} h_{2,j,r}(T_n(X_{k})) + 
\left(
\lambda \left( \frac{1}{r} - 1 \right)
\sum_{l=1}^{d} \gamma_l 
\right)
\end{equation}
over the event 
$\cap_{j=1}^d \{X_k \in \partial_1, k \neq \sigma_{j-1,j}\}$;
this and \eqref{e:skpp1} imply
\[
{\mathbb E}[S_{k+1}' | {\mathscr F}_k] 
\le 
 S_{k}'  +
r^n
\left(
\lambda \left( \frac{1}{r} - 1 \right)
\sum_{l=1}^{d} \gamma_l 
\right).
\]
Moving the last expression to the left of the inequality sign
and
subtracting
$k 
\left(
\lambda \left( \frac{1}{r} - 1 \right)
\sum_{l=1}^{d} \gamma_l 
\right)
$ from both sides give
${\mathbb E}[S_{k+1} | {\mathscr F}_k] \le S_{k}$ over the same event.
It remains to show 
\begin{equation}\label{e:lastcase}
{\mathbb E}[S_{k+1} | {\mathscr F}_k] \le S_{k}
\end{equation}
over the event  
$\cup_{j=1}^d \{X_k \in \partial_1,  k = \sigma_{j-1,j} \}.$ 
In this case
\[
 S_{k}'=\Gamma_{j-1} h_{2,j-1,r}(T_n(X_{k})),
S_{k+1}' = \Gamma_j h_{2,j,r}(T_n(X_{k+1})),
\]
for some $j \in \{1,2,3,...,d\}$ and $X_k \in \partial_j.$
By Lemma \ref{l:inequalityoverjumps} 
\[
S_{k}'=\Gamma_{j-1} h_{2,j-1,r}(T_n(X_{k})) \ge
\Gamma_{j} h_{2,j,r}(T_n(X_{k}));
\]
this and \eqref{e:inforp1} imply \eqref{e:lastcase}.
\end{proof}
The upperbound
on
${\mathbb P}_{x}(\sigma_{d-1,d} < \tau_n < \tau_0)$
now follows from the supermartingale constructed above:
\begin{proof}[Proof of Proposition \ref{p:upperboundonlong}]
To use the supermartingale $S$ to bound 
${\mathbb P}_{x}(\sigma_{d-1,d} < \tau_n < \tau_0) < \rho^{n(1-\epsilon))}$
we need to truncate time by an application of the following
fact (see \cite[Theorem A.1.1]{thesis}):
there exists $c_1 > 0$ and $n_0 > 0$ such that
${\mathbb P}_x(\tau_n \wedge \tau_0 > c_1n)  \le \rho^{2n}$
for $n > n_0$.
Although they give the
same results, the truncation argument varies in \cite{sezer2018approximation,unlu2019excessive,kabran2020approximation};
below we closely follow the one given in \cite{kabran2020approximation}.
We decompose
${\mathbb P}_{x}(\sigma_{d-1,d} < \tau_n < \tau_0)$:
\begin{equation}\label{e:decomposesup}
{\mathbb P}_{x}(\sigma_{d-1,d} < \tau_n < \tau_0)
\le
{\mathbb P}_{x}(\sigma_{d-1,d} < \tau_n < \tau_0 \le c_6 n) + 
\rho^{2n}.
\end{equation}
To bound the last probability we apply the optional sampling theorem
to the supermartingale $S$ of \eqref{e:defsupS} at the bounded terminal time
$\eta = c_6 n \wedge \tau_0 \wedge \tau_n$:
\begin{align}
r^n = S_0 &\ge {\mathbb E}_{x}[ S_{\eta}]\notag\\
     &=	   {\mathbb E}_{x}\left[ S'_\eta - \eta \left( \lambda \left(\frac{1}{r} -1\right) \sum_{j=1}^d \gamma_j \right) r^n \right]\notag \\
	&\ge 
    	   {\mathbb E}_{x}\left[ S'_\eta\right] 
- c_6 n\left( \lambda \left(\frac{1}{r} -1\right) \sum_{j=1}^d \gamma_j \right) r^n \label{e:suparg1}
\end{align}
$S' > 0$ implies
\begin{equation}\label{e:suparg2}
{\mathbb E}_x[S'_\eta] \ge 
{\mathbb E}[S'_\eta {\bm 1}_{\{\sigma_{d-1,d} < \tau_n < \tau_0 \le c_6n\}}].
\end{equation}
Over the event $\{\sigma_{d-1,d} < \tau_n < \tau_0 \le c_6n \}$ we have:
\begin{align*}
\eta &= \tau_n, \\
S'_\eta &= \Gamma_d h_{2,d,r}(T_n(X_{\tau_n}))  
=
\Gamma_d \sum_{j=1}^d \gamma_j h_{j,r}(T_n(X_{\tau_n}))
\ge \Gamma_d \gamma_d h_{d,r}(T_n(X_{\tau_n})) = \Gamma_d \gamma_d.
\end{align*}
This, \eqref{e:suparg1} and \eqref{e:suparg2} imply
\[
\frac{1}{\gamma_d \Gamma_d} 
r^n \left( 1 + c_6 n\left( \lambda \left(\frac{1}{r} -1\right) \sum_{j=1}^d \gamma_j \right)\right)\ge 
{\mathbb P}_x(\sigma_{d-1,d} < \tau_n < \tau_0 \le c_6n).
\]
This inequality holds for any $r \in (\rho,1)$; it follows that it also holds
for $\rho$, i.e., 
\[
\frac{1}{\gamma_d \Gamma_d} 
\rho^n \left( 1 + c_6 n\left( \lambda \left(\frac{1}{r} -1\right) \sum_{j=1}^d \gamma_j \right)\right)\ge 
{\mathbb P}_x(\sigma_{d-1,d} < \tau_n < \tau_0 \le c_6n).
\]
The statement of the proposition follows from this and \eqref{e:decomposesup}.
\end{proof}

\subsection{Completion of the analysis} 
As the last step, we derive a lower bound on ${\mathbb P}_{x}(\tau_n < \tau_0).$
Following \cite{unlu2019excessive,kabran2020approximation} we will do this via subharmonic functions. Define
\[
x\in {\mathbb Z}_+^d \mapsto g_{i,n}(x) = h_{i,\rho_i}(T_n(x))= \rho_i^{n- \sum_{j=1}^i x(j)}, k=1,2,3,...d,
\]
and
\begin{equation}\label{e:defgn}
g_n(x) \doteq \max_{i \in \{ 1,2,3,...,d\}} g_{i,n}(x).
\end{equation}
\begin{proposition}\label{p:lowerboundP}
\begin{equation}\label{e:lowerboundP}
g_n(x)-\rho^n \le {\mathbb P}_x(\tau_n < \tau_0).
\end{equation}
\end{proposition}
\begin{proof} 
That $g_{i,n}(x) = h_{i,\rho_i}(T_n(x))$
and the calculation in the proof of Proposition \ref{p:hkr} give
\[
{\mathbb E}_x[g_{i,n}(X_1)] -g_{i,n}(x) =
g_{i,n}(x) \left( \lambda (\frac{1}{\rho_i}-1) + \mu_k (\rho_i-1) {\bm 1}_{\partial_k^c}(x)\right).
\]
The right side of this equality is $0$ for $x \in \partial_k^c$ and positive for $x \in \partial_k.$
It follows that $g_{i,n}$ is $X$-subharmonic on ${\mathbb Z}_+^d.$
Therefore, $k\mapsto g_{i,n}(X_k)$ is a submartingale. The stability of $X$ implies that $\tau_0 < \infty$ almost surely.
This, that 
$k\mapsto g_{i,n}(X_k)$ is a submartingale and the optional sampling theorem give
\begin{align*}
g_{i,n}(x) &\le {\mathbb E}[ h_{i,n}(X_{\tau_n})1_{\{\tau_n < \tau_0\}}] + {\mathbb E}[g_{i,n}(0)1_{\{\tau_n > \tau_0\}}]
\intertext{ $g_{i,n} \le 1$ on $\partial A_n$ and $g_{i,n}(0) = \rho_i^n$ imply}
&\le {\mathbb P}_x(\tau_n < \tau_0) + \rho_i^n.
\end{align*}
Applying $\max_{i\in \{1,2,3,...,d\}}$ to both sides gives \eqref{e:lowerboundP}.
\end{proof}
Define the order relation $\preccurlyeq$ on the nodes $\{1,2,3,...,d\}$ as follows: 
$i \preccurlyeq j$ if $\rho_i \le \rho_j$ and $i \le j.$ It follows from its definition
that $\preccurlyeq$ is a partial order relation (it is reflexive, antisymmetric, transitive).
Define
\[
{\mathscr M} \doteq \{ i \in \{1,2,3,...,d\}: 
\nexists j \in \{1,2,3,..,d\} \text{ such that } \rho_j \ge \rho_i \text{ and } j > i \}.
\]
The set ${\mathscr M}$ consists exactly of the maximal elements of the relation
$\preccurlyeq$.
$d$ is the maximum of $\{1,2,3,...,d\}$, therefore there can be 
no $j \in \{1,2,3,...,d\}$ satisfying $j > d$, this implies that 
$d \in {\mathscr M}$
always holds, in particular, ${\mathscr M}$ is never empty.
A similar argument implies
$\rho_i \neq \rho_j$ for $i,j \in {\mathscr M}$. Let us label members of ${\mathscr M}$ by
$i_1$, $i_2$,..., $i_{|A|}$ so that 
\begin{equation}\label{e:rhos}
\rho_{i_1} > 
\rho_{i_2} > 
\rho_{i_3} > \cdots >\rho_{i_{|{\mathscr M}|}};
\end{equation}
$i_1  < i_3 <\cdots < i_{|{\mathscr M}|}$ and $i_{|{\mathscr M}|} = d$  once again follow 
from the definitions just given.

The point $x \in A_n$ must satisfy $g_n(x) < \rho^n$ for 
the bound \eqref{e:lowerboundP} to be nontrivial. 
The next proposition identifies the set of such $x$.
\begin{align}
R_\rho &\doteq  \label{d:Rrho}
\bigcap_{i \in {\mathscr M} } \left\{x \in {\mathbb R}^d_+: \sum_{j=1}^{i} x(j) \le  \left(1 - \frac{\log \rho}{\log \rho_i}\right) \right\}\\
R_{\rho,n} &\doteq \{x \in {\mathbb Z}_+^d: x/n \in R_\rho \} \notag \\
\bar{R}_{\rho,n} &\doteq \notag
\bigcup_{i \in {\mathscr M} } \left\{x \in {\mathbb Z}^d_+: \sum_{j=1}^{i} x(j) \ge 1+ n\left(1 - \frac{\log \rho}{\log \rho_i}\right) \right\}.
\end{align}
Two comments: 1) $x \in \bar{R}_{\rho,n}$ if 
$\sum_{j=1}^{\min {\mathscr M} } x(j) \ge 1$;
in particular $\bar{R}_{\rho,n} = 
\{x \in{\mathbb Z}_+^d: \sum_{j=1}^{d} x(j) \ge 1$ if
${\mathscr M} = \{d\}$ (i.e., if $\rho_d \ge \rho_i$ for all $i$);
2) $\bar{R}_{\rho,n}$ is almost the complement of $R_{\rho,n}$.

\begin{proposition}\label{p:aboutgn}
The following hold:
\begin{equation}\label{e:gnstatements}
g_n(x) = \max_{i \in {\mathscr M}} \rho_i^{n -\sum_{j=1}^i x_j}, \min_{x \in A_n} g_n(x) = \rho^n,
\end{equation}
\begin{equation}\label{e:minimizersofgn}
\{g_n(x) = \rho^n\} =  R_{\rho,n}.
\end{equation}
and
\begin{equation}\label{e:boundonratio}
\frac{g_n(x)}{g_n(x) - \rho^n} \le \frac{1}{1-\rho},
\end{equation}
for $x \in \bar{R}_{\rho,n}.$
\end{proposition}
\begin{proof}
For $i \neq j$,
the definitions of $\preccurlyeq$, $g_{i,n}$ and $g_{j,n}$ imply
\[
g_{i,n}(x) \le g_{j,n}(x),
\]
if $i \preccurlyeq j$. Therefore, one can replace the index set $\{1,2,3,...,d\}$ in \eqref{e:defgn} 
with ${\mathscr M}$; this implies the first statement in \eqref{e:gnstatements}.

Reversing the order of $\min$ and $\max$ gives
\[
\min_{x \in A_n} g_n(x) = \min_{x \in A_n} \max_{i \in \{1,2,3,...,d\}}g_{i,n}(x)
\ge \max_{i \in \{1,2,3,...,d\}} \min_{x \in A_n} g_{i,n}(x).
\]
By the definition of $g_{i,n}$ we have
\[
\min_{x \in A_n} g_{i,n}(x) = \rho_i^n.
\]
The last two displays imply
\[
\min_{x \in A_n}g_n(x) \ge \rho^n.
\]
On the other hand, $g_{n}(0) = \max_{i \in \{1,2,3,...d\}}  \rho_i^n = \rho^n$; this
and the last display imply the second statement in \eqref{e:gnstatements}.

Once again, the definition of $g_{i,n}$ implies
\[
\{x: g_{i,n}(x) \le \rho^n \} =  \left\{x \in {\mathbb Z}^d_+: \sum_{j=1}^{i} x(j) \le n \left(1 - \frac{\log \rho}{\log \rho_i}\right) \right\}.
\]
This and $g_n(x) = \max_{i \in {\mathscr M}} \rho_i^{n -\sum_{j=1}^i x_j}$ implies 
\eqref{e:minimizersofgn}.

Finally, if $x\in {\mathbb Z}_+^d$ satisfies $\sum_{j=1}^{i} x(j) \ge 1+ n\left(1 - \frac{\log \rho}{\log \rho_i}\right)$ we have
$g_{i,n}(x) \ge \rho^n/ \rho_i.$ Then
\[
\frac{g_n(x)}{g_n(x) - \rho^n}
\le 
\frac{g_{i,n}(x)}{g_{i,n}(x) - \rho^n}
\le
\frac{\rho^n/\rho_i}{\rho^n/\rho_i - \rho^n} \le \frac{1}{1-\rho_i}
\le \frac{1}{1-\rho},
\]
which implies \eqref{e:boundonratio}.
\end{proof}
Recall the convention \eqref{e:rhos}; i.e., $\rho_{i_1} = \rho$ and $i_1 = \min{\mathcal M}_{\preccurlyeq}$.
Therefore,  
$R_{\rho} \subset \{x \in {\mathbb R}_+^d: \sum_{j=1}^{i_1} x(j) = 0 \}$;
in particular $R_\rho$ has strictly lower dimension than $d$. If 
$|{\mathscr M}| =1$, i.e., if $\rho_d > \rho_i$ for all $i < d$ we
have $R_{\rho} =\{0\}.$

Define
\[
g: {\mathbb R}_+^d \mapsto {\mathbb R}, g(x) = \max_{i \in \{1,2,..,d\}}
(1-\sum_{j=1}^i x(j))\log_{\rho}\rho_i
\]
The following lemma follows from the definition of $g$ and the arguments
of the previous proposition:
\begin{lemma}\label{l:gnasg}
$g_n(x) = \rho^{ng(x/n)}$, 
$\max_{x \in A} g(x) =1$, $\{x: g(x) = 1 \} = R_\rho.$
\end{lemma}
The upperbound on the approximation error follows from the
bounds above:
\begin{theorem}\label{t:convergence}
 For  $\epsilon > 0$
there exists $n_0 > 0$ such that
\begin{equation}\label{e:mainrelativeerrorbound}
\frac{ |{\mathbb P}_{x}(\tau_n < \tau_0) - {\mathbb P}_{T_n(x)}(\tau < \infty)|}{ {\mathbb P}_{x}(\tau_n < \tau_0)} \le \rho^{n(1 - g(x/n)-\epsilon)}
\end{equation}
for all $n > n_0$ and for any $x \in \bar{R}_{\rho,n}$.
In particular, for $x_n/n \rightarrow x  \in A - R_\rho$ the relative error
decays exponentially with rate $-\log(\rho)(1-g(x))>0$, i.e.,
\begin{equation}\label{e:decayrate}
\liminf_n -\frac{1}{n} \log \left(
\frac{ |{\mathbb P}_{x_n}(\tau_n < \tau_0) - {\mathbb P}_{T_n(x_n)}(\tau < \infty)|}{ {\mathbb P}_{x_n}(\tau_n < \tau_0)} \right)
 \ge -\log(\rho)(1-g(x)).
\end{equation}
\end{theorem}
\begin{proof}
By Lemma \ref{l:boundondifference} 
\begin{equation}\label{e:boundonnum}
|{\mathbb P}_{x_n}(\tau_n < \tau_0) - {\mathbb P}_{T_n(x_n)}(\tau < \infty)|
\le 
{\mathbb P}_x(\bar\tau_0 < \bar\tau_n < \infty)  + 
{\mathbb P}_{x}(\sigma_{d-1,d} < \tau_n < \tau_0)
\end{equation}
By Proposition \ref{p:upperboundonlong} we can choose $n_0$ large enough
so that
\[
{\mathbb P}_{x}(\sigma_{d-1,d} < \tau_n < \tau_0) < 
\rho^{n(1-\epsilon/2)}
\]
for $n > n_0.$ This, \eqref{e:boundonnum}
 and Proposition \ref{p:boundonbtau0btaun} give
\[
|{\mathbb P}_{x_n}(\tau_n < \tau_0) - {\mathbb P}_{T_n(x_n)}(\tau < \infty)|
\le 
\rho^{n(1-\epsilon/2)}
+
\rho^n \frac{1}{\gamma_d} \sum_{j=1}^d \gamma_j.
\]
for $n > n_0.$  This and the lowerbound on ${\mathbb P}_x(\tau_n < \tau)$
given in Proposition \ref{p:lowerboundP} imply
\begin{align*}
\frac{ |{\mathbb P}_{x}(\tau_n < \tau_0) - {\mathbb P}_{T_n(x)}(\tau < \infty)|}{ {\mathbb P}_{x}(\tau_n < \tau_0)} 
&\le\frac{1}{ g_n(x)-\rho^n}
\left( \rho^{n(1-\epsilon/2)} +
\rho^n \frac{1}{\gamma_d} \sum_{j=1}^d \gamma_j\right)
\intertext{and by Proposition \ref{p:aboutgn} and Lemma \ref{l:gnasg}}
&\le
\frac{1}{1-\rho}\rho^{n g(x/n)}
\left( \rho^{n(1-\epsilon/2)} +
\rho^n \frac{1}{\gamma_d} \sum_{j=1}^d \gamma_j\right)
\end{align*}
Finally, increase $n_0$ if necessary so that
$ \frac{1}{1-\rho} \left( 1 + \frac{1}{\gamma_d}\sum_{j=1}^d \gamma_j\right)< 
\rho^{-n\epsilon/2}$
for $n > n_0$; then the last display and this choice of $n_0$ imply
\eqref{e:mainrelativeerrorbound}; \eqref{e:decayrate} follows
from \eqref{e:mainrelativeerrorbound}, the continuity of $g$ and
from the fact that $\epsilon > 0$ can be chosen arbitrarily small.
\end{proof}

\section{$Y$-harmonic functions from harmonic systems}
\label{s:dg2}
This section provides a framework for
the construction of $Y$-harmonic functions. This can be done without
additional effort
for constrained random walks arising from any
Jackson network and we will do so. The main element of the 
framework
is the reduction of
the construction of $Y$-harmonic functions
to the solution of certain equations represented
by graphs with labeled edges.
In the next section we will provide a solution to these
equations for the case when the constrained random walk represents
a tandem network.

In this section we allow ${\mathcal V}$ to be the set of possible
increments of any constrained random walk arising from a Jackson
network, i.e.:
\[
{\mathcal V} = \{- e_i + e_j, i,j \in \{0,1,2,...,d\}, i \neq j\},
\]
where $e_0 = 0 \in {\mathbb Z}^d$;
the unconstrained increments of $X$ takes values in ${\mathcal V}$
with probabilities ${\mathbb P}(I_k = -e_i +e_j) = p(i,j)$
where $p \in {\mathbb R}_+^{(d+1) \times (d+1)}$, $p(i,i) =0$,
$i \in \{0,1,2,3,...,d\}$ and
$\sum_{i,j=0}^d p(i,j) = 1.$
With this update to the set of possible increments the definition
of $X$ remains unchanged.
The increment $-e_i +e_j$ represents a customer leaving node $i$
and joining node $j$ where node $0$ represents outside of the system.
For a general Jackson network the total service rates are defined as
\[
\mu_i = \sum_{j=0}^d p(i,j), i \in \{1,2,3,...,d\}.
\]
The $Y$ process is defined as in \eqref{e:defY}
on $\partial_1$ with possible increments
\begin{align*}
{\mathcal V}_Y & = \{ {\mathcal I}_1 v, v \in {\mathcal V}\}\\
 & = \{ v_{1,j} \doteq e_1 +e_j, v_{i,1} \doteq -e_i -e_1, 
v_{i,j} \doteq -e_i + e_j, i,j \in \{0,1,2,3...,d\},i \neq j\}
\end{align*}
\begin{equation}\label{d:ygeneral}
Y_{k+1} = Y_k + \pi_1(Y_k, J_k).
\end{equation}

For $\alpha \in {\mathbb C}^{d-1}$ we will index the components
of the vector $\alpha$ with the set 
$\{2,3,4,...,d\}$, i.e.,
$\alpha = (\alpha(2),\alpha(3),...,\alpha(d))$ (so, more precisely,
$\alpha \in {\mathbb C}^{\{2,3,4,...d\}}$).
The class of $Y$-harmonic functions we seek are to be linear
combinations of functions of the form
\begin{align*}
y \in {\mathbb Z}^d &\mapsto [(\beta,\alpha),y],\\
[(\beta,\alpha),y] &\doteq \beta^{y(1)- \sum_{j=2}^d y(j) } 
\prod_{j=2}^d \alpha(j)^{y(j)}.
\end{align*}
$[(\beta,\alpha),y]$ is $\log$-linear in $y$, 
i.e., $y\mapsto \log([(\beta,\alpha),y])$ is 
linear in $y$.

For $a \subset \{2,3,...,d\}$ and $a^c = \{0,1,2,3,...,d\} - a$,
define the characteristic polynomial
\begin{equation}\label{e:hamiltonianba}
{\mathbf p}_a(\beta,\alpha)
\doteq 
\left( \sum_{ i \in a^c, j=0  }^d p(i,j) 
[(\beta,\alpha) ,v_{i,j} ]
+ 
\sum_{i \in a} \mu_i\right)
\end{equation}
the characteristic equation
\begin{equation}\label{e:chareqgenerald}
{\mathbf p}_a(\beta,\alpha) = 1,
\end{equation}
and the characteristic surface
\[
{\mathcal H}_a \doteq \{ (\beta,\alpha) \in {\mathbb C}^d: 
{\mathbf p}_a(\beta,\alpha) = 0\}
\]
of the boundary $\partial_a$, $ a \subset \{2,3,4,...,d\}.$
We will write ${\mathbf p}$ instead of ${\mathbf p}_\emptyset$.
${\mathbf p}_a$ is not a polynomial but a rational function;
to make it a polynomial one must multiply it by 
$\beta \prod_{j=2}^d \alpha(j)$; nonetheless, to keep our 
language simple
we will refer to the rational
\eqref{e:hamiltonianba} as the ``characteristic polynomial.''

Conditioning $Y$ on its first step gives
\begin{lemma}\label{l:singleterm1}
Suppose $(\beta,\alpha) \in {\mathcal H}$.
Then $[ (\beta,\alpha),\cdot]$ is $Y$-harmonic on 
$\Omega_Y - \bigcup_{j=2}^d \partial_j$.
\end{lemma}
\begin{proof}
For $y \in \Omega_Y-\bigcup_{j=2}^d \partial_j$ we have
\[
{\mathbb E}_y[ [ (\beta,\alpha), Y_1]] - [(\beta,\alpha),y] =
[(\beta,\alpha),y]
( {\mathbf p}(\beta,\alpha) -1 ) = 0,
\]
where the last equality follows from $(\beta,\alpha) \in {\mathcal H}.$
\end{proof}

Define the operator $D_a$ acting on functions on 
${\mathbb Z}^d$ and giving functions on $\partial_a$:
\begin{align*}
&D_aV = g,~~~~ V: {\mathbb Z}^d \rightarrow {\mathbb C},\\
& g( y)
\doteq
\left(\sum_{i \in a} \mu_i V(y)  +  
\sum_{ i\in a^c, j=0}^d p(i,j) V( y + v_{i,j}) \right)
- V(y);
\end{align*}
\begin{lemma}\label{l:Daeq0}
 $D_a V = 0$ 
if and only if $V$ is $Y$-harmonic on $\partial_a$.
\end{lemma}
The proof follows from the definitions.
Define
\begin{equation}\label{e:Cjbetaalpha}
C(i,\beta,\alpha) \doteq 
\mu_i
-\sum_{j=0}^d p(i,j) [ (\beta,\alpha),v_{i,j} ].
\end{equation}
\begin{lemma}
For $y \in \partial_a$ and $(\beta,\alpha) \in {\mathcal H}$:
\begin{equation}\label{e:imageofD2shortgd}
D_i( [(\beta,\alpha,\cdot)])(y) =C(i,\beta,\alpha) [(\beta,\alpha),y].
\end{equation}
\end{lemma}
\begin{proof}
\begin{align*}
D_i( [(\beta,\alpha,\cdot)])(y) &=
\left( \mu_i
+\sum_{i'\neq i, j=0}^d p(i',j) [ (\beta,\alpha),v_{i',j} ]  -1\right) 
[ (\beta,\alpha), y]\\
&=
\left( \mu_i
-\sum_{ j=0}^d p(i,j) [ (\beta,\alpha),v_{i,j} ] \right) 
[ (\beta,\alpha), y],
\end{align*}
where we used $1={\mathbf p}(\beta,\alpha)$.
\end{proof}

\begin{lemma}\label{l:singleterm2}
Suppose $(\beta,\alpha) \in {\mathcal H} \cap {\mathcal H}_i$. Then
$[(\beta,\alpha),\cdot]$ is $Y$-harmonic on 
$\Omega_Y-\bigcup_{j \in \{2,3,...,d\} - \{i\}} \partial_j$.
\end{lemma}
\begin{proof}
$(\beta,\alpha) \in {\mathcal H}$ and Lemma \ref{l:singleterm1} imply
that $[(\beta,\alpha),\cdot]$ is $Y$-harmonic on 
$\Omega_Y-\bigcup_{j=2}^d \partial_j.$
That ${\mathbf p}(\beta,\alpha) = 1$ and ${\mathbf p}_i(\beta,\alpha) =1$
imply
\[
C(i,\beta,\alpha) = 
{\mathbf p}_i(\beta,\alpha)-{\mathbf p}(\beta,\alpha) =0.
\]
This and the last two lemmas imply that $[(\beta,\alpha),\cdot ]$ is
$Y$-harmonic on $\partial_i.$
\end{proof}

For $\alpha \in {\mathbb C}^{\{2,3,...,d\}}$ and
$j in \{2,3,4,...,d\}$ define
$\alpha\{j\} \in {\mathbb C}^{\{2,3,...d\}}$
as follows: 
\[
\alpha\{j\}(i) = \begin{cases}
 1,&\text{ if } i = j\\
 \alpha(i),&\text{ otherwise.}
\end{cases}
\]
For example, for $d=4$,
$j = 4$
and $\alpha = (0.2,0.3,0.4)$, $\alpha\{4\} = (0.2,0.3,1).$

For $i \in \{2,3,4,...,d\}$, multiplying both sides of the characteristic
equation ${\bf p}(\beta,\alpha) =0$
by
$\alpha(i)$
gives a second order polynomial equation in $\alpha(i)$: denote
the roots by $r_1$ and $r_2$. From the coefficients of the second
order polynomial we read
\begin{equation}\label{e:conjinr}
r_1 r_2 = 
\frac{ 
\sum_{j=0}^d p(i,j) [ (\beta, \alpha\{i\}), v_{i,j}]}
{\sum_{j=0}^d p(j,i) [ (\beta, \alpha\{i\}), v_{j,i}]}.
\end{equation}
From these two roots we get two points $(\beta,\alpha_1),
(\beta,\alpha_2)$ on ${\mathcal H}$
whose components are
\[
\alpha_1(i) =r_1, \alpha_2(i)=r_2, 
\]
and
\begin{equation}\label{e:conjugacydg2}
\alpha_1(j) =\alpha_2(j) = \alpha(j), j\neq i.
\end{equation}
By \eqref{e:conjinr}
\begin{equation}\label{e:conjalpha}
\alpha_1(i)\alpha_2(i)
=
\frac{ 
\sum_{j=0}^d p(i,j) [ (\beta, \alpha\{i\}), v_{i,j}]}
{\sum_{j=0}^d p(j,i) [ (\beta, \alpha\{i\}), v_{j,i}]}.
\end{equation}
If $\alpha_1(i)\neq \alpha_2(i)$
we call $(\beta,\alpha_1) \neq (\beta,\alpha_2) \in {\mathcal H}$
$i$-conjugate.
Note that $\alpha_1\{i\} = \alpha_2\{i\} = \alpha\{i\}$; therefore
\eqref{e:conjalpha} can also be written as
\begin{equation}\label{e:conja1toa2}
\alpha_1(i)\alpha_2(i)
=
\frac{ 
\sum_{j=0}^d p(i,j) [ (\beta, \alpha_1\{i\}), v_{i,j}]}
{\sum_{j=0}^d p(j,i) [ (\beta, \alpha_1\{i\}), v_{j,i}]}.
=
\frac{ 
\sum_{j=0}^d p(i,j) [ (\beta, \alpha_2\{i\}), v_{i,j}]}
{\sum_{j=0}^d p(j,i) [ (\beta, \alpha_2\{i\}), v_{j,i}]}.
\end{equation}

Next proposition generalizes \cite[Proposition 4]{sezer2018approximation}
to the current setup. 

\begin{proposition}\label{p:harmonicYtwoterms}
Suppose that 
$(\beta,\alpha_1)$
and $(\beta,\alpha_2)$ are $i$-conjugate and $C(i,\beta,\alpha_j)$,
$j=1,2$ are well defined.
Then
\begin{equation*}
h_\beta \doteq C(i,\beta,\alpha_2) [(\beta,\alpha_1),\cdot]
- C(i,\beta,\alpha_1) [(\beta,\alpha_2),\cdot]
\end{equation*}
is $Y$-harmonic on $\partial_i$.
\end{proposition}
\begin{proof}
The definition \eqref{e:Cjbetaalpha}  of $C$, \eqref{e:imageofD2shortgd} and 
linearity of $D_i$ imply
\begin{align*}
D_i(h_\beta) &=
C(i,\beta,\alpha_2) C(i,\beta,\alpha_1) 
[(\beta,\alpha_1),\cdot]
- 
C(i,\beta,\alpha_2) C(i,\beta,\alpha_1) 
[(\beta,\alpha_2),\cdot]
\intertext{ \eqref{e:conjugacydg2} implies
$[(\beta,\alpha_1),z] = [(\beta,\alpha_2),z]$ for $z \in \partial_i$
and therefore the last line reduces to}
&=0.
\end{align*}
Lemma \ref{l:Daeq0} now implies that $h_\beta$ is $Y$-harmonic
on $\partial_i$.
\end{proof}

The class of $Y$-harmonic functions we identify in this section
is based on graphs with labeled edges; let us now 
give a precise definition of these.
We denote any graph by its adjacency matrix $G$; the structure of $G$ is as follows.
Let 
$V_G$, a finite set, 
denote the set of vertices of $G$; let $L$ denote the set of labels.
For two vertices $i \neq j$, $G(i,j) = 0$ if they are
disconnected, and $G(i,j) = l$ if an edge with label 
$l \in L$ connects them; such an edge will be called an $l$-edge. 
As usual, an edge from a vertex to itself is called a loop.
For a vertex $j \in V_G$, 
$G(j,j)$ is the set of the labels of the loops on $j$. 
Thus $G(j,j)\subset L$ 
is set
valued.

In graph theory a graph is said to be $k$-regular if all of its vertices
have the same degree (number of edges) $k$ \cite[page 5]{diestelgt}. 
We generalize this definition
as follows:
\begin{definition}
Let $G$ and $L$ be as above.
If each vertex $j \in V_G$ has
a unique $l$-edge (perhaps an $l$-loop) 
for all $l \in L$ we will call $G$ $L$-regular.
\end{definition}

\begin{definition}\label{d:Yharmonic} A $Y$-harmonic system consists
of a $\{2,3,4,...,d\}$-regular graph $G$,
the variables
$ (\beta,\alpha_j)  \in {\mathbb C}^d$, ${\boldsymbol c}_j \in {\mathbb C}$,  $j \in V_G$
and these equations/constraints:
\begin{enumerate}
\item
$(\beta,\alpha_j) \in {\mathcal H}, {\boldsymbol c}_j \in {\mathbb C}-\{0\}, j \in V_G$, 
\item
$\alpha_{i} \neq \alpha_{j}$,  if $i \neq j, i,j  \in V_G$,
\item
$\alpha_{i}, \alpha_{j}$  are $G(i,j)$-conjugate if  
$ G(i,j) \neq 0$, $i \neq j, i,j \in V_G$,
\item 
\begin{equation}
\label{e:correctmultipliers}
{\boldsymbol c}_{i}/{\boldsymbol c}_{j} = 
-\frac{C(G(i,j),\beta,\alpha_j)}{C(G(i,j),\beta,\alpha_i)}, \text{ if } G(i,j) \neq 0,
\end{equation}
\item
$(\beta,\alpha_j) \in {\mathcal H}_l \text{ for all } l \in G(j,j), j \in V_G.$
\end{enumerate}
\end{definition}

\begin{proposition}\label{p:simpleharmonicfunctions}
Suppose that 
a $Y$-harmonic system with graph $G$
has a solution $({\boldsymbol c}_j, (\beta,\alpha_j), j \in V_G)$.
Then
\begin{equation}\label{e:defhG}
h_G \doteq \sum_{j \in V_G} {\boldsymbol c}_j [(\beta,\alpha_j),\cdot]
\end{equation}
is $Y$-harmonic.
\end{proposition}
In the proof the following decomposition is useful:
for $y \in \partial_a$ and $(\beta,\alpha) \in {\mathcal H}$:
\begin{align}
D_a( [(\beta,\alpha,\cdot)])(y) &=
\left( \sum_{i \in a} \mu_i
+\sum_{i \in a^c, j} p(i,j) [ (\beta,\alpha),v_{i,j} ]  -1\right) 
[ (\beta,\alpha), y]\notag \\
&=
\left( \sum_{i \in a} \mu_i
-\sum_{i \in a, j=0}^d p(i,j) [ (\beta,\alpha),v_{i,j} ] \right) 
[ (\beta,\alpha), y]\notag \\
&=
\sum_{i \in a} \left(
\mu_i - \sum_{ j=0}^d p(i,j) [ (\beta,\alpha),v_{i,j} ] \right) 
[ (\beta,\alpha), y]\notag \\
&= \sum_{i \in a} D_i( [(\beta,\alpha,\cdot)])(y) \label{e:Dadecomposed}.
\end{align}

\begin{proof}[Proof of Proposition \ref{p:simpleharmonicfunctions}]
By Lemma \ref{l:singleterm1},
all summands of $h_G$ are $Y$-harmonic on $\Omega_Y-\bigcup_{j=2}^d \partial_j$ because $(\beta,\alpha_j)$,$ j \in V_G$ are all on the characteristic surface
${\mathcal H}$.
It remains to show that $h_G$ is $Y$-harmonic on all $\partial_a \cap \Omega_Y$
$a \subset \{2,3,4,...,d\}$ and $a \neq \emptyset.$ We will do this
by induction on $|a|$. Let us start with $|a|=1$, i.e., $a = \{l\}$, for 
some $l \in \{2,3,4,..,d\}$
Take any vertex $i \in V_G$; if $l \in G(i,i)$
then $(\beta,\alpha_i) \in {\mathcal H}_l$ and by Lemma \ref{l:singleterm2}
$[(\beta,\alpha_i),\cdot]$ is $Y$-harmonic on $\partial_l$.
Otherwise, the definition of a harmonic system implies that
there exists a unique vertex $j$ of $G$
such that $G(i,j) = l$. This implies, by definition, that
 $(\beta,\alpha_i)$ and $(\beta,\alpha_{j})$
are $l$-conjugate and by Proposition \ref{p:harmonicYtwoterms} 
and \eqref{e:correctmultipliers} 
\[
{\boldsymbol c}_i [(\beta,\alpha_i),\cdot] + {\boldsymbol c}_{j} [ (\beta,\alpha_{j}),\cdot]
\]
is $Y$-harmonic on $\partial_l$. 
Thus, all summands
of $h_G$ are either $Y$-harmonic on $\partial_l$ 
or form pairs which are so;
this implies that the sum $h_G$ is $Y$-harmonic on $\partial_l$.

Now assume $h_G$ is $Y$-harmonic for all
$a'$ with $|a'| = k-1$; fix an $a \subset \{2,3,4,...,d\}$
such that $|a| = k$ and a $i \in a$;
by \eqref{e:Dadecomposed}
\[
D_a( h_G) = D_{a-\{i\}}(h_G) + D_i(h_G).
\]
The induction assumption and Lemma \ref{l:Daeq0} imply 
that the first term on the right is zero; the same lemma
and the previous paragraph imply the same for the second term.
Then $D_a(h_G) = 0$; this and Lemma \ref{l:Daeq0} finish the
proof of the induction step.
\end{proof}

\subsection{Simple extensions}
In this subsection we show how the solution of a harmonic system 
for a lower dimensional process can provide solutions for a related harmonic
system of a higher dimensional process provided that the higher
dimensional process is a ``simple extension'' (defined below) 
of the lower dimensional one. 

For two integers $d_2 > d_1 > 0 $ let 
$p_i \in {\mathbb R}^{(d_i + 1) \times (d_i +1)}$ ,$i=1,2$, be two transition
matrices.
Define
$p' \in {\mathbb R}^{ (d_1 +1) \times (d_1+1)}$
as
\begin{align}
p'(i,j) &= p_2(i,j) \label{e:defppij}
\intertext{if $i\in \{0,1,2,3,...,d_1\}$, $j \in \{1,2,3,...,d_1\}$ and}
\label{e:defppi0}
p'(i,0) &= p_2(i,0) + \sum_{j=d_1+1}^{d_2}
p_2(i,j), 
~~~i\in \{1,2,3,...,d_1\}.
\end{align}
\begin{definition}\label{d:simpleextension}
We say that $p_2$ is a simple extension of $p_1$ if 
\begin{align}
&p' =  \left(\sum_{i,j=0}^{d_1} p'(i,j)\right) p_1 , p' \neq 0,
\label{d:pandp1}
\\
&p_2(i,j) = 0 \text{ if }i \in \{d_1 +1,...,d_2\},
j \in \{1,2,3,...,d_1\}.
\label{e:constraintonp1}
\end{align}
\end{definition}

An example:
\begin{equation}\label{ex:simpleextension}
p_1 = \left( \begin{matrix}
	0 & 1/7 & 0 \\
	0 &  0      & 4/7 \\
	2/7 & 0   & 0 
\end{matrix}
\right), 
p_2 = \left( \begin{matrix}
	  0       & 0.05   & 0     & 0 &  0.02\\
	  0       & 0         & 0.2 & 0 & 0         \\
	  0.1 & 0         & 0    & 0.1 & 0  \\
	 0        & 0         & 0    & 0        & 0.25\\
	 0.1 & 0        & 0    & 0.18  & 0 
	\end{matrix}
	\right).
\end{equation}
Figure \ref{f:simpleextensionextex} shows the topologies of the networks
corresponding to $p_1$ and $p_2.$

\begin{figure}[h]
\begin{center}
\scalebox{0.8}{
\centerline{\input{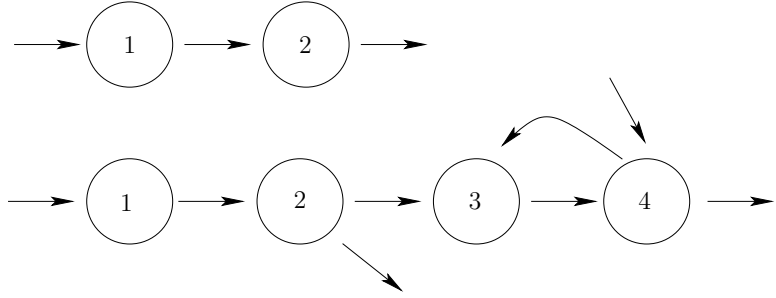}}}
\end{center}

\vspace{-0.65cm}
\caption{\hspace{0.25cm}Networks corresponding to $p_1$ and $p_2$ of \eqref{ex:simpleextension}, 
second is a simple extension of the first\label{f:simpleextensionextex}}
\end{figure}

\begin{definition}\label{d:simpleextensionG}
Let $G$ be a $L$-regular.
Let ${L}_1 \supset L$ be another set of labels.
$G$'s simple extension $G_1$ to an $L_1$-regular graph
is defined as follows: $V_{G_1} = V_G$
and
\begin{align}
G_1(i,j) &= G(i,j),  i \neq j,  i,j\in V_G \label{d:edgecomplete1}\\
G_1(j,j) &= G(j,j) \cup (L_1 - L), j \in V_G.\notag
\end{align}
\end{definition}
To get $G_1$ from $G$ one adds to each vertex of $G$ an $l$-loop for each
$l \in L_1 - L$. $G$ is $L$-regular implies that $G_1$ is $L_1$-regular.
Figure \ref{f:simpleextendGtex} gives an example.

\begin{figure}[h]
\begin{center}
\scalebox{0.8}{
\centerline{\input{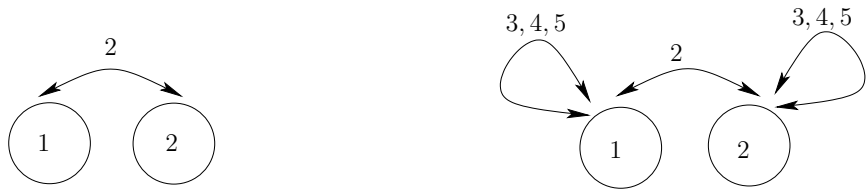}}}
\end{center}

\vspace{-0.65cm}
\caption{\hspace{0.25cm}A
$\{2\}$-regular graph and its  simple extension 
to a $\{2,3,4,5\}$-regular graph\label{f:simpleextendGtex}}
\end{figure}

If $Y^2$ is a simple
extension of $Y^1$, any solution to a $Y^1$-harmonic system
implies a related solution to a related $Y^2$-harmonic system:
\begin{proposition}\label{p:invarianceundersimpleextensions}
For $d_2 > d_1 > 1$ let 
$p_i \in {\mathbb R}_+^{(d_i + 1) \times (d_i + 1)}$, $i=1,2$ be
transition matrices such that $p_2$ is a simple extension of $p_1.$
Let $Y^i$ be defined through \eqref{d:ygeneral} with $d=d_i$, $i=1,2$
and $p = p_i$, $i=1,2.$
Let $G_i$, $i=1,2$ be
 $\{2,3,...,d_i\}$-regular graphs for $Y^2$ and $Y^1$ such that $G_2$
is a simple extension of $G_1$ (in the sense of Definition \ref{d:simpleextensionG}).
Suppose $(\beta,\alpha_k), {\boldsymbol c}_k, k \in V_{G_1}$
solve the harmonic system associated with $G_1$. 
For $k \in V_{G_2} = V_{G_1}$ define
$\alpha^2_k \in {\mathbb C}^{d_2 + 1}$ as follows
\begin{align}
\label{e:defa1p1}
\alpha^2_k(j)  &= \alpha_k^1(j),~ j\in \{2,3,4,...,d_1\} \\
\label{e:defa1p2}
\alpha^2_k(j)&= \beta, j \in \{d_1+1,d_1+2,...,d_2\}.
\end{align}
Then $(\beta,\alpha^2_k), {\boldsymbol c}_k, k \in V_{G_2}$ 
solves the harmonic system
defined by  $G_2$ and $p_2.$
\end{proposition}
The definition \eqref{e:defa1p2} extends
$\alpha^1_k \in {\mathbb C}^{d_1-1}$ to $\alpha^2_k \in {\mathbb C}^{d_2-1}$
by assigning the value $\beta$
to the additional dimensions of
$\alpha^2_k$.
This, \eqref{d:pandp1} and \eqref{e:constraintonp1} imply that, when
$\alpha^2_k$ is defined as above, the
harmonic system defined by $G_2$ reduces to that defined by $G_1$;
the details are as follows:
\begin{proof}
By assumption, $(\beta,\alpha_k^1)$, ${\boldsymbol c}_k$,
$k \in V_{G_0}$ satisfy the five conditions
listed under Definition \ref{d:Yharmonic} for $G=G_1$ and $p=p_1$.
We want to show that this implies that
the same holds for 
$(\beta,\alpha^2_k), {\boldsymbol c}_k, k \in V_{G_2}$  
for $G=G_2$ and $p=p_2.$ 

Fix any $k \in V_{G_1}$;
\eqref{e:defa1p1} and \eqref{e:defa1p2} imply
\begin{equation}\label{e:reduction1}
[(\beta,\alpha^2_k),v_{i,j}^2 ] 
=
\begin{cases} [(\beta,\alpha_k^1),v_{i,j}^1 ], &\text{ if }j \le d_1,\\
 [(\beta,\alpha_k^1),v_{i,0}^1 ], &\text{ if }j > d_1,
\end{cases}
\end{equation}
for all $i \in \{0,1,2,...,d_1\}$, $j \in \{0,1,2,...,d_2\}$, $i\neq j.$
Similarly,
\eqref{e:defa1p2} implies
\begin{equation}\label{e:ba1v1p3}
[(\beta,\alpha^2_k), v^2_{i,j} ] = 1
\end{equation}
for all
$i,j \in \{0,d_1 + 1,d_1 +2,...,d_2\}$, $i\neq j.$

Let ${\mathbf p}^2$ denote the characteristic polynomial of $Y^2$ and 
let ${\mathcal H}^2$ denote its characteristic surface; we would like
to show $(\beta,\alpha_k^2) \in {\mathcal H}^2$, i.e., 
${\mathbf p}^2(\beta,\alpha_k^2) = 1.$

By \eqref{e:constraintonp1}
\begin{align}\label{e:computePY1}
{\mathbf p}^{2}(\beta,\alpha^2_k) &= 
\sum_{i=0,j=1}^{d_1}
p_2(i,j) [(\beta,\alpha^2_k), v^2_{i,j} ] + 
\sum_{i=1,j \in\{0,d_1+1,...,d_2\}}^{d_1}
p_2(i,j) [(\beta,\alpha^2_k), v^2_{i,j} ] \\
&~~~+ 
\sum_{i,j \in\{0,d_1+1,...d_2\}}
p_2(i,j) [(\beta,\alpha^2_k), v^2_{i,j} ].\notag
\intertext{ \eqref{e:defppij}, \eqref{e:reduction1} and \eqref{e:ba1v1p3} imply}
&=
\sum_{i=0,j=1}^{d_1}
p'(i,j) [(\beta,\alpha_k^1), v^1_{i,j} ] + 
\sum_{i=1, j \in \{0,d_1 +1,...,d_2\} }^{d_1} \label{e:threetermsum}
p_2(i,j) [(\beta,\alpha_k^1), v^1_{i,0} ] \\
&~~~+ 
\sum_{i,j \in \{0,d_1+1,...,d_2\}} p_2(i,j)
\notag
\end{align}
\eqref{e:defppi0} implies
that the second sum above equals
$\sum_{i=1}^{d_1} p'(i,0) [(\beta,\alpha_k^1), v^1_{i,0}].$
Substitute this back in \eqref{e:threetermsum} to get
\begin{align*}
{\mathbf p}^{2}(\beta,\alpha^2_k) &= 
\sum_{i,j=0}^{d_1}
p'(i,j) [(\beta,\alpha_k^1), v_{i,j} ] + 
\sum_{i,j \in\{0,d_1+1,...,d_2 } p_2(i,j)
\intertext{which, by \eqref{d:pandp1}, equals}
& = \left(\sum_{i,j =0 }^{d_1} p'(i,j) \right)
\sum_{i,j=0}^{d_1}p_1(i,j)[(\beta,\alpha_k^1), v^{1}_{i,j} ] 
+ 
\sum_{i,j \in \{0,d_1+1,...,d_2\}} p_2(i,j)
\intertext{$(\beta,\alpha_k) \in {\mathcal H}$, 
\eqref{e:defppij}, \eqref{e:defppi0} and \eqref{e:constraintonp1}
now give}
& = \sum_{i,j=0}^{d_2} p_2(i,j) =1
\end{align*}
i.e., $(\beta,\alpha_k^2) \in {\mathcal H}^2.$
This proves $(\beta,\alpha^2_k) \in {\mathcal H}^{2}$, $k \in V_{G_2}$,i.e.,
the first part of Definition \ref{d:Yharmonic} is satisfied by
$(\beta,\alpha^2_k), {\boldsymbol c}_k, k \in V_{G_2}$  
for $G=G_2$ and $p=p_2.$ 

By definition $\alpha_{i}^1 \neq \alpha_{j}^1$ for
$i \neq j$, this and \eqref{e:defa1p1} imply $\alpha_{i}^2 \neq
\alpha_{j}^2$, i.e, the second part of Definition \ref{d:Yharmonic}
also holds for
$(\beta,\alpha^2_k), {\boldsymbol c}_k, k \in V_{G_2}$  
for $G=G_2$ and $p=p_2.$ 

Let us now show that the third part of the same definition is also satisfied.
Fix any $i\neq j$ with $G_2(i,j) =l \in \{2,3,4,....,d_1\}$
(that $G_2$ is a simple extension of $G_1$ means that
$G_2(i,j) \in \{2,3,4,..,d_1\}$; see \eqref{d:edgecomplete1}).
We want to show that $(\beta,\alpha_{i}^2)$ and
$(\beta,\alpha_{j}^2)$ are $l$-conjugate,
, i.e., that they satisfy
\eqref{e:conjugacydg2} and \eqref{e:conja1toa2}:
\begin{equation}\label{e:conjugacydg2Y1}
\alpha_{i}^2(k) = \alpha_j^2(k), k \neq l,
\end{equation}
\begin{equation}\label{e:conja1toa2Y1}
\alpha_{i}^2(l) \alpha_j^2(l)
=\frac{\sum_{k=0}^{d_2} p_2(l,k) [(\beta,\alpha_i^2\{l\}),v_{l,k}^2]} 
{\sum_{k=0}^{d_2} p_2(k,l) [(\beta, \alpha_i^2\{l\}), v_{k,l}^2]}.
\end{equation}
By definition
$G_2(i,j)=l$
when $G_1(i,j) =l$; $G_1(i,j)=l$ implies that
$\alpha_i^1$ and $\alpha_j^1$ are $l$-conjugate;  in particular,
they satisfy \eqref{e:conjugacydg2}.
\eqref{e:conjugacydg2Y1} 
follows from this, \eqref{e:defa1p1} and \eqref{e:defa1p2}.

We next prove \eqref{e:conja1toa2Y1}.
For $l \in \{2,3,4,...d_1\}$,
$\alpha_j^2(l) = \alpha^1_j(l)$ and $\alpha_i^2(l) = \alpha^2_i(l)$;
therefore 
$\alpha_i^2(l) \alpha_j^2(l) = \alpha^1_i(l)\alpha^1_j(l)$.
$\alpha^1_i$ and $\alpha^1_j$ are $l$-conjugate, in particular,
they satisfy \eqref{e:conja1toa2}:
\[
\alpha_{i}^1(l) \alpha_j^1(l)
=\frac{\sum_{k=0}^{d_1} p_1(l,k) [(\beta,\alpha_i^1\{l\}),v_{l,k}^1]} 
{\sum_{k=0}^{d_1} p_1(k,l) [(\beta, \alpha_i^1\{l\}), v_{k,l}^1]}.
\]
Then to prove \eqref{e:conja1toa2Y1} it suffices to prove
\begin{equation}\label{e:ratiosimplified}
\frac{\sum_{k=0}^{d_2} p_2(l,k) [(\beta,\alpha_i^2\{l\}),v^2_{l,k}]} 
{\sum_{k=0}^{d_2} p_2(k,l) [(\beta, \alpha_i^2\{l\}), v^2_{k,l}]}
=\frac{\sum_{k=0}^{d_1} p_1(l,k) [(\beta,\alpha_i^1\{l\}),v^1_{l,k}]} 
{\sum_{k=0}^{d_1} p_1(k,l) [(\beta, \alpha_i^1\{l\}), v^1_{k,l}]}.
\end{equation}
This follows from a decomposition parallel to the one given in
\eqref{e:computePY1}; let us first apply it to the numerator:
\begin{align}
\sum_{k=0}^{d_2}& p_2(l,k) [(\beta,\alpha_i^2\{l\}),v_{l,k}^2] 
=\sum_{k=0}^{d_1} p_2(l,k) [(\beta,\alpha_i^2\{l\}),v_{l,k}^2]
+
\sum_{k=d_1+1}^{d_2} p_2(l,k) [(\beta,\alpha_i^2\{l\}),v_{l,k}^2]
\notag
\intertext{\eqref{e:defppij}, \eqref{e:defppi0}, \eqref{d:pandp1} and \eqref{e:reduction1}  imply}
&~~~=\sum_{k=1}^{d_1} p'(l,k) [(\beta,\alpha_i^1\{l\}),v_{l,k}^1]
+
p_2(l,0)  [(\beta,\alpha_i^1\{l\}),v_{l,0}^1]
+
\sum_{k={d_1+1}}^{d_2} p_2(l,k) [(\beta,\alpha_i^1\{l\}),v_{l,0}^1]
\notag \\
&~~~= \left(\sum_{ i,j=0 }^{d_1} p'(i,j) \right) 
\sum_{k=0}^{d_1} p_1(k,l) [(\beta,\alpha_i^1\{l\}),v_{k,l}^1].\label{e:numeratorsimplified}
\end{align}
A parallel argument for the denominator gives (this time also using
\eqref{e:constraintonp1})
\[
\sum_{k=0}^{d_2} p_2(k,l) [(\beta,\alpha_i^2\{l\}),v^2_{k,l}] 
=
 \left(\sum_{i,j=0}^{d_1} p'(i,j) \right) \sum_{k=0}^{d_1}
p_1(k,l)[(\beta,\alpha_i^1\{l\}),v_{k,l}^1].
\]
Dividing \eqref{e:numeratorsimplified} by the last equality gives
\eqref{e:ratiosimplified}.

The proof that parts 4-5 of Definition \ref{d:Yharmonic}
hold for 
$(\beta,\alpha^2_k), {\boldsymbol c}_k, k \in V_{G_2}$  
for $G=G_2$ and $p=p_2$ 
is parallel to the arguments just given and is omitted.
\end{proof}

In the following remark we note several facts that we don't need
directly in our arguments. Their proofs
are very similar to the arguments given above and are left to the
reader:
\begin{remark}
{\em
Let $Y^1$ and $Y^2$ be as above, i.e, $Y^i$ is $d_i$ dimensional, $d_1 < d_2$
and $Y^2$ is a simple extension of $Y^1$;  for $y \in {\mathbb Z}^{d_2}$,
let $y^{1,d_1}$ be denote the projection of $y$ onto its first $d_1$ 
coordinates. 
If $h$ is $Y^1$-harmonic then, $y\mapsto h(y^{1,d_1})$ is $Y_2$-harmonic.
Similarly, let $G_1$ and $G_2$, ${\boldsymbol c}_k$, $\alpha_k^i$,
$k \in V_{G_1}$ be 
as in the previous proposition; then $h_{G_2}(y) = h_{G_1}(y^{1,d_1})$.
}
\end{remark}

\subsection{$\partial B$-determined $Y$-harmonic functions}
A $Y$-harmonic function $f$ is said to be
$\partial B$-determined if
\[
f(y) = {\mathbb E}[f(Y_{\tau}) 1_{\{\tau < \infty \}}], y \in \Omega_Y.
\]
$y \mapsto {\mathbb P}_y(\tau < \infty)$ is the unique $\partial B$-determined
$Y$-harmonic function with the value $1$ on $\partial B$.
The next proposition identifies simple conditions under which
a $Y$-harmonic function defined by a harmonic system is $\partial B$-determined.

\begin{proposition}\label{p:balayagesimpled}
Let $(\beta,\alpha_j), {\boldsymbol c}_j$ be the solutions of a $Y$-harmonic system
with its graph $G$ and let $h_G$ be defined as in \eqref{e:defhG}.
If 
\[
|\beta| < 1,~~~|\alpha_j(i)|, \le 1, i=2,3,...,d, j \in V_G,
\]
then $h_G$ is $\partial B$-determined.
\end{proposition}
The proof is identical to that of \cite[Proposition 5]{sezer2018approximation};
for ease of reference we give an outline below:
\begin{proof}
Define $\xi_n = \inf\{k: Y_k(1) = \sum_{j=2}^d Y_k(j) + n \}.$
The optional sampling theorem and the fact that $h_G$ is $Y$-harmonic
imply
\[
h_G(y) = {\mathbb E}[ h_G(Y_\tau) 1_{\{\tau \le \xi_n\}}]
+ {\mathbb E}[h_G(Y_{\xi_n}) 1_{\{\xi_n \le \tau\}}].
\]
$|\alpha_i| \le 1$ implies
$|{\mathbb E}[h_G(Y_{\xi_n}) 1_{\{\xi_n \le \tau\}}]| \le \beta^n |V_G|\max_{j\in V_G} |{\boldsymbol c}_j|$. This, $|\beta| < 1$ and letting 
$n\rightarrow \infty$ in the last display give
\[
h_G(y) = {\mathbb E}[ h_G(Y_\tau) 1_{\{\tau < \infty \}}].
\]

\end{proof}

\section{Harmonic systems for constrained random walks representing
tandem networks and the computation of ${\mathbb P}_y(\tau < \infty)$}
\label{s:hstq}
Throughout this section we will denote the dimension of the system
with ${\boldsymbol d}$; the arguments below for
${\boldsymbol d}$ dimensions require the consideration
of all walks with dimension $d \le {\boldsymbol d}$.

We will now define a specific sequence of regular graphs
for tandem walks
and construct a particular solution to the harmonic system defined
by these graphs.
These particular solutions
will give us an exact formula  for ${\mathbb P}_y(\tau < \infty)$ in terms of
the superposition of a finite number of $\log$-linear $Y$-harmonic functions.

We will assume 
\begin{equation}\label{e:separatemu}
\mu_i \neq \mu_j, i \neq j;
\end{equation}
this generalizes $\mu_1 \neq \mu_2$ assumed in \cite{sezer2018approximation}.
One can treat parameter values which violate \eqref{e:separatemu} by taking limits
of the results of the present section, we give several examples in Section \ref{s:conclusion}.

The characteristic polynomials for the tandem walk are:
\begin{align}\label{e:charpoltandem}
{\mathbf p}(\beta,\alpha) &= 
\lambda \frac{1}{\beta} + 
\mu_1 \alpha(2) + \sum_{j=2}^{{\boldsymbol d}} \mu_j\frac{\alpha(j+1)}{\alpha(j)}\\
{\mathbf p}_i(\beta,\alpha) &= 
\lambda \frac{1}{\beta} + 
\mu_1 \alpha(2) + \mu_i +  \sum_{j=2, j \neq i}^{{\boldsymbol d}} \mu_j\frac{\alpha(j+1)}{\alpha(j)},
\notag
\end{align}
where by convention $\alpha({\boldsymbol d}+1) = \beta$ (this convention
will be used throughout this section, and in particular, 
in Lemma \ref{l:conditionforintersection}, \eqref{e:conja1toa2t}, and
\eqref{e:Cjbetaalphatandem}).

\eqref{e:charpoltandem} implies
\begin{lemma}\label{l:conditionforintersection}
For $j \in \{2,3,4,...,{\bm d} \}$,
$(\beta,\alpha) \in {\mathcal H} \cap {\mathcal H}_j$
$\iff$
$\mu_j \frac{\alpha(j+1)}{\alpha(j)} = \mu_j$
$\iff$
$\alpha(j+1) = \alpha(j)$.
\end{lemma}

For the tandem walk, the conjugacy relation \eqref{e:conjalpha}
reduces to
\begin{equation}\label{e:conja1toa2t}
\alpha_1(i)\alpha_2(i)
=
\begin{cases}
\frac{\alpha(3)\mu_2}{\mu_1}, & i =2, \\
\frac{\alpha(i-1)\alpha(i+1)\mu_i}{\mu_{i-1}}, &i=2,3,...,{\bm d}.
\end{cases}
\end{equation}

For tandem walks the functions $C(j, \beta,\alpha)$ of 
\eqref{e:Cjbetaalpha} reduce to 
\begin{equation}\label{e:Cjbetaalphatandem}
C(j,\beta,\alpha) = 
\mu_j - \mu_j \frac{ \alpha(j+1)}{\alpha(j)},
\end{equation}

We define $\{2,3,...,{\bm d}\}$-regular graphs $G_{d,{\bm d}}$, $ d \in \{1,2,3,...,{\boldsymbol d}\}$ as follows:
\begin{equation}\label{d:defVGd}
V_{G_{d,{\bm d}}} = \{ a\cup\{ d\}, a \subset \{1,2,3,...,d-1\}\};
\end{equation}
for $j \in (a \cup \{d\})$, $j \neq 1$, define $G_{d,{\bm d}}$ by
\begin{align}\label{d:defGd}
G_{d,{\bm d}}( a\cup\{d\}, a\cup \{d\} \cup \{j-1\}) &= j \text{ if } j-1 \notin a 
\end{align}
and
\begin{equation}\label{e:Gdloops}
G_{d,{\bm d}}( a\cup \{d\}, a \cup \{d\}) = \{2,3,4,...,d\}- a \cup \{d\};
\end{equation}
these and its symmetry determine $G_{d,{\bm d}}$ completely.
We note that vertices of $G_{d,{\bm d}}$ are subsets
of $\{1,2,3,...,{\bm d}\}$; we will assume
these sets to be sorted, for $a \subset \{1,2,3,...,{\bm d}\}$,
$a(1)$ denotes the smallest element of $a$, $|a|$ the number of elements
in $a$ and $a(|a|)$ the greatest element of $a$.
Figure \ref{f:G4tex} shows the graph $G_{4,4}$.

\begin{figure}[h]
\begin{center}
\scalebox{0.8}{
\centerline{\input{G4tex}}}
\end{center}

\vspace{-0.65cm}
\caption{\hspace{0.25cm}$G_{d,{\bm d}}$ for $d={\bm d}=4$\label{f:G4tex}}
\end{figure}

The next two propositions follow directly from the above definition:
\begin{proposition}\label{p:Gsimpleextensiontandem}
$G_{d,{\bm d}}$ 
is the simple extension of $G_{d,d}$
to a $\{2,3,...,{\bm d}\}$-regular graph.
\end{proposition}
Let
$G_{d+1,{\bm d}}^{k}$ denote the subgraph of $G_{d+1,{\bm d}}$ 
consisting of the vertices $\{a, k,d+1\}$, 
$a\subset \{1,2,3,...,k-1\}$. 
\begin{proposition}\label{p:embedding}
One can represent $G_{d+1,{\bm d}}$ as a disjoint union of the
graphs $G_{k,{\bm d}}, k=1,2,..,d,$ and the vertex $\{d+1\}$ as follows:
for $a \subset \{1,2,3,...,k-1\}$ map the vertex
$a \cup \{k\}$ of $G_{d+1,{\bm d}}$ to $a \cup \{k, d+1\}$.
This maps $G_{k,{\bm d}}$ to the 
subgraph $G_{d+1,{\bm d}}^{k}$ of $G_{d+1,{\bm d}}$ 
consisting of the vertices $a \cup\{ k,d+1\}$, 
$a\subset \{1,2,3,...,k-1\}$. 
The same map preserves 
the edge
structure of $G_{k,{\bm d}}$ as well except for the $d+1$-loops. These loops
on $G_{k,{\bm d}}$ are broken and are mapped to $d+1$-edges between
$G^{k}_{d+1,{\bm d}}$ and $G^{d-1}_{d+1,{\bm d}}$. 
\end{proposition}
Figure \ref{f:G4tex} shows an example of the
decomposition described in Proposition \ref{p:embedding}.

For $a \subset \{2,3,4,...,{\bm d}\}$ define
\begin{align}\label{e:solutiontandem}
{\boldsymbol c}^*_{a} &\doteq (-1)^{|a|-1} \prod_{j=1}^{|a|-1} \prod_{l = a(j) + 1}^{a(j+1)}
\frac{\mu_l - \lambda}{\mu_l-\mu_{a(j)}}\\
\alpha^*_{a}(l) &\doteq
\begin{cases}
	1 &\text{ if } l \le a(1)\\
	\rho_{a(j)}, &\text{ if } a(j) < l \le a(j+1),\\
	\rho_{a(|a|)} & \text{ if } l > a(|a|),
\end{cases}
\label{e:defalphastar}\\
\beta^*_{a}&\doteq \rho_{a(|a|)}, \label{e:defbetastar}
\end{align}
$l \in \{2,3,...,{\bm d}\}$ (remember that we assume sets ordered
and $a(|a|)$ denotes
the largest element in the set). Let us give several examples to these
definitions for ${\boldsymbol d} = 8$:
\begin{align}\label{e:examplescstaralphastar}
{\boldsymbol c}^*_{\{5\}}&= 1\notag \\
\alpha^*_{\{5\}} &= (1,1,1,1,\rho_5,\rho_5,\rho_5)\notag \\
{\boldsymbol c}^*_{\{3,6\}} &= - \frac{ \mu_4 -\lambda}{\mu_4 - \mu_3} \frac{\mu_5 -
\lambda}{\mu_5-\mu_3} \frac{\mu_6-\lambda}{\mu_6-\mu_3}.\notag \\
{\boldsymbol c}^*_{\{3,5,7\}} &
= (-1)^2 \frac{\mu_4 - \lambda}{\mu_4 -\mu_3} \frac{\mu_5-\lambda}{\mu_5-\mu_3}
\frac{\mu_6-\lambda}{\mu_6-\mu_5} \frac{\mu_7-\lambda}{\mu_7-\mu_5}\notag \\
\alpha^*_{\{3\}} &= (1,1,\rho_3,\rho_3,\rho_3,\rho_3,\rho_3)\\
\alpha^*_{\{3,6\}} &= (1,1,\rho_3,\rho_3,\rho_3,\rho_6,\rho_6)\notag \\
\alpha^*_{\{3,5,7\}} &= (1,1,\rho_3,\rho_3,\rho_5,\rho_5,\rho_7)\notag \\
\alpha^*_{\{8\}} &= (1,1,1,1,1,1,1);\notag
\end{align}
remember that we index the components of $\alpha^*$ with 
$\{2,3,4,...,{\bm d}\}$; therefore,
e.g., the first $1$ on the right side of the last line is $\alpha^*_{\{8\}}(2).$

It follows from \eqref{e:solutiontandem} 
and \eqref{e:defalphastar} that
\begin{align*}
{\boldsymbol c}^*_{a \cup \{d_1,d_2\}} &=-{\boldsymbol c}^*_{a \cup \{d_1\}} 
\prod_{l=d_1+1}^{d_2} \frac{ \mu_l - \lambda}{\mu_l - \mu_{d_1}}
\\
\alpha^*_{a \cup \{d_1\}} &= \alpha^*_{a \cup \{d_1,{\boldsymbol d}\}}
\end{align*}
for any 
$1 < a(|a|) < d_1 < d_2 \le {\bm d}$ and 
$a \subset \{2,3,4,...,{\bm d}\}$;
These and Proposition \ref{p:embedding} imply
\begin{proposition}
For $d < {\boldsymbol d}$ and $y \in \partial B$
\begin{equation}\label{e:embeddingfunc}
-\left(\prod_{l=d+1}^{{\boldsymbol d}} \frac{ \mu_l - \lambda}{\mu_l - \mu_d}\right)
\sum_{ a \in V_{G_{d,{\bm d}} }} {\boldsymbol c}_a^* [ (\beta_a^*, \alpha_a^*), y ]
=
\sum_{ a \in V_{G_{{\bm d},{\boldsymbol d}}^d }} {\boldsymbol c}_a^* [ (\beta_a^*, \alpha_a^*), y ]
\end{equation}
\end{proposition}

\begin{proposition}\label{p:provesolution}
For $d \le \boldsymbol d$, let
$G_{d,{\bm d}}$ be as in \eqref{d:defVGd} and \eqref{d:defGd}. Then
$(\beta^*_{a\cup\{d\}}, \alpha^*_{a\cup\{d\}})$,
${\boldsymbol c}^*_{a\cup\{d\}}$, 
 $a \subset \{1,2,3,...,d-1\}$, defined in \eqref{e:solutiontandem},
solve the harmonic system defined
by $G_{d,{\bm d}}$.
\end{proposition}
\begin{proof}
A $d'$ tandem walk is a simple extension of the tandem walk defined
by its first $d'-1$ dimensions. This, 
Propositions \ref{p:Gsimpleextensiontandem}, \ref{p:invarianceundersimpleextensions}
and the definitions of $\beta^*$ and $\alpha^*$ above
imply that 
it suffices to prove the current proposition only
for $d=\boldsymbol d$.

The vertices of $G_{{\bm d},{\bm d}}$ are $a \cup \{{\bm d}\}$,
$a \subset \{1,2,3,....,{\bm d}-1\}$, and for all of them
we have $\beta^*_{a\cup\{{\bm d}\}} = \rho_{\bm d}$ by 
definition \eqref{e:defbetastar}.
Let us begin by showing 
$\left(\rho_{\boldsymbol d}, \alpha^*_{a\cup\{\boldsymbol d\}}\right)$, 
$a \subset \{1,2,3,....,{\bm d}-1\}$ is on the characteristic surface
${\mathcal H}$ of the tandem walk. We will write $\alpha^*$ instead
of $\alpha^*_{a\cup\{\boldsymbol d\}}$, the set $a$ will be clear from context.

Let us first consider the case when $a(1) > 1$, i.e., when $1 \notin a$;
the opposite case is treated similarly and is left to the reader.
Then $\alpha^*(l) = 1$ for $2 \le l \le a(1)$.
By definition $\alpha^*(i) = \alpha^*(i+1)$ if $a(j) < i < a(j+1)$;
these and $\beta^*_{a\cup\{{\boldsymbol d}\}} = \rho_{\boldsymbol d}$ give
\begin{align*}
{\mathbf p}(\rho_{\boldsymbol d},\alpha^*) &= 
\mu_{\boldsymbol d} + 
\sum_{j=1}^{a(1)-1} \mu_j  + \mu_{a(1)} \rho_{a(1)}  +
 \sum_{j \in (a^c -\{1\cdots a(1)-1 \})} \mu_j\\
&~~~~~+ \sum_{j \in (a-\{a(1)\}) } \mu_j \frac{\alpha^*(j+1)}{\alpha^*(j)}
+ \rho_{\boldsymbol d} \frac{\mu_{\boldsymbol d}}{\alpha^*(\boldsymbol d)}
\intertext{(where
$a^c= \{1,2,3,...,{\bm d}-1\} - a$)
and 
 in the last expression we have used the convention
$\alpha^*({\boldsymbol d} + 1) =\beta^*$;
by definition \eqref{e:defalphastar} 
$\alpha^*(a(j+1)) = \rho_{a(j)}$,
$\alpha^*(a(j)) = \rho_{a(j-1)}$ and therefore}
&= 
\mu_{\boldsymbol d} + 
\sum_{j=1}^{a(1)-1} \mu_j  + \lambda +  
 \sum_{j\in (a^c -\{1\cdots a(1)-1 \})} \mu_j\\
&~~~~~+ \sum_{j=2}^{|a|} \mu_{a(j)} \frac{\rho_{a(j)}}{\rho_{a(j-1)}}
+ \mu_{a(|a|)}
\intertext{
$\mu_{a(j)} \rho_{a(j)}/\rho_{a(j-1)} = \mu_{a(j-1)}$ implies}
&= 
\mu_{\boldsymbol d} + 
\sum_{j=1}^{a(1)-1} \mu_j  + \lambda +  
 \sum_{j\in (a^c -\{1\cdots a(1)-1\})} \mu_j\\
&~~~~~+ \sum_{j=2}^{|a|} \mu_{a(j-1)} + \mu_{a(|a|)}\\
&= 
\mu_{\boldsymbol d} + 
\sum_{j=1}^{a(1)-1} \mu_j  + \lambda +  \hspace{-0.5cm}
 \sum_{j\in (a^c -\{1\cdots a(1)-1 \})} \mu_j+ \sum_{j \in a }\mu_j = 1;
\end{align*}
i.e., $(\rho_{\boldsymbol d}, \alpha^*) \in {\mathcal H}.$

If $a_1 \neq a_2$ 
take any $ i \in a_1 -a_2$ (relabel the sets 
if necessary so that  $a_1 -a_2 \neq  \emptyset$).
Let $j$ be the index of $i$ in $a_1$, i.e., $a_1(j) = i$.
Then by definition, $\alpha^*_{a_1\cup\{\boldsymbol d\}}(j+1) = \rho_i$; but
$i \notin a_2$ and \eqref{e:separatemu} imply that no component
of $\alpha^*_{a_1\cup \{\boldsymbol d\}}$ equals $\rho_i$, and therefore
$\alpha^*_{a_1\cup\{\boldsymbol d\}} \neq \alpha^*_{a_2\cup\{\boldsymbol d\}}$. 
This shows that
$\alpha^*_{a \cup \{\boldsymbol d\}}$, $a \subset \{1,2,...,{\bm d}-1\}$ satisfy
the second part of Definition \ref{d:Yharmonic}.

Fix a vertex $a \cup \{\boldsymbol d\}$ of $G_{{\bm d},{\boldsymbol d}}$. 
By definition,
for each of its elements $l$,
this vertex is connected to $ a \cup \{\boldsymbol d\} \cup \{l-1\}$ if
$l-1 \notin a$ or 
or to $ a \cup \{\boldsymbol d\} - \{l-1\}$ if $l-1 \in a$.
Then to show that 
$(\beta^*_{a \cup \{{\bm d} \}}, \alpha^*_{a \cup \{\boldsymbol d\}})$, 
$a \subset \{1,2,3,....,{\bm d}-1\}$
satisfy the third part of Definition \ref{d:Yharmonic} it suffices to 
prove that for each $a \subset \{1,2,3,...,{\bm d}-1\}$,
and each $l \in a \cup \{\boldsymbol d\}$ such that $l-1 \notin a \cup \{\boldsymbol d\}$
$\alpha^*_{a \cup \{\boldsymbol d\}}$ and
$\alpha^*_{a \cup \{\boldsymbol d\} \cup \{l-1\}}$ are $l$-conjugate.
For ease of notation let us denote 
$a \cup \{l-1\}$ by $a_1$,
$\alpha^*_{a\cup \{\boldsymbol d\}}$ by 
$\alpha^*$, $\alpha^*_{a_1 \cup \{\boldsymbol d\}}$ by $\alpha^*_1$ and
 $\beta^*_{a_1} =\beta^*_a$  by $\beta^*$ 
(because we have assumed $d={\boldsymbol d}$,
both 
$\beta^*$ and $\beta^*_1$ are equal to $\rho_{\boldsymbol d}$).
We want to show that $(\beta^*, \alpha^*)$ and $(\beta^*, \alpha^*_1)$ are $l$-conjugate.
Let us assume $2 < l < {\boldsymbol d}$, the cases $l=2,{\boldsymbol d}$ 
are treated almost the same way
and are left to the reader.
By assumption $l \in \alpha^*$ but $l-1 \notin \alpha^*$.
If $l$ is the $j^{th}$ element of $a$, i.e.,
$l = a(j)$; then $a(k) =a_1(k)$ for $k < j$,
$a_1(j)=l-1$, $a(k-1) = a_1(k)$ for $k > j$.
This and the definition \eqref{e:defalphastar} of $\alpha^*$ imply 
\begin{equation}\label{e:therestsame}
\alpha^*(i) = \alpha^*_1(i), i \in \{2,3,4,...,{\bm d} \}, i \neq l,
\end{equation}
i.e., $\alpha^*$ and $\alpha^*_1$ satisfy \eqref{e:conjugacydg2}
(for example, for $\boldsymbol d=8$, $\alpha^*_{\{3,6\}}$ is given in 
\eqref{e:examplescstaralphastar}; on the other hand 
$\alpha^*_{\{3,5,6\}} = (1,1,\rho_3,\rho_3,\rho_5,\rho_6,\rho_6)$ and
indeed 
$ \alpha^*_{\{3,6\}}(i) = \alpha^*_{\{3,5,6\}}(i), i \neq 6$).
Definition
\eqref{e:defalphastar}
also implies 
\begin{equation}\label{e:alphastarl}
\alpha^*_1(l) = \rho_{a_1(j)} = \rho_{l-1},~ \alpha^*_1(l+1) = 
\rho_{a_1(j+1)} = \rho_l,
\end{equation}
On the other hand,
again by \eqref{e:defalphastar}, and by $l-1 \notin a$,
we have
\begin{equation*}
\alpha^*(l) = \alpha^*(l-1) = \rho_{a(j-1)} \text{ and }  
\alpha^*(l+1) = \rho_l.
\end{equation*}
Then
\[
\frac{1}{\alpha^*(l)} \frac{\alpha^*(l-1) \alpha^*(l+1) \mu_l}{\mu_{l-1}}
= \rho_{l-1}
\]
and, by \eqref{e:alphastarl} this equals $\alpha^*_1(l)$, i.e.,
$\alpha^*_1$ and $\alpha^*$ satisfy \eqref{e:conja1toa2t}.
This and \eqref{e:therestsame} mean that 
$(\beta^*,\alpha_1^*)$ and $(\beta^*,\alpha)$ are
$l$-conjugate.

Now we will prove that the 
${\boldsymbol c}^*_{a\cup\{\boldsymbol d\}}$, 
$a \subset \{2,3,4,...,{\bm d}-1\}$ defined in \eqref{e:solutiontandem}
satisfy the fourth part of Definition \ref{d:Yharmonic}. 
The structure of $G_{{\boldsymbol d},{\boldsymbol d}}$ implies 
that it suffices to check that 
\begin{equation}\label{e:conditiononCs}
\frac{{\boldsymbol c}^*_{a}}{{\boldsymbol c}^*_{a_1} } = 
-\frac{C(l',\rho_{\boldsymbol d}, \alpha^*_{a_1})}
{C(l',\rho_{\boldsymbol d},\alpha^*_a)}
\end{equation}
holds for any $ l' \in a$ such that $l'-1 \notin a$  and $a_1 = a \cup \{l'-1\}.$
There are three cases to consider: $l'=2$, $l'=\boldsymbol d$ 
and $2 < l' < \boldsymbol d$; we will only treat the last.
For $2 < l' < \boldsymbol d$ one needs to further 
consider the cases $a(1) = l'$ and $a(1) < l'$.
For $b \subset \{2,3,4,...,{\bm d}-1\}$, 
${\boldsymbol c}^*_{b\cup\{\boldsymbol d\}}$ of \eqref{e:solutiontandem}
is the product of a parity term and a running product of 
$d-b(1)$ ratios of the form $(\mu_l - \lambda)/(\mu_l - \mu_{b(j)}).$
The ratio of the parity terms of $a$ and $a_1$ is $-1$ because $a_1$ has one additional term.
If $a(1) = l'$ then $a_1(1) = l'-1$ and 
the only difference between the running products
in the definitions of ${\boldsymbol c}^*$ and ${\boldsymbol c}^*_1$
is that the latter has an additional initial term
$(\mu_{l'}-\lambda)/(\mu_{l'}-\mu_{l'-1})$ and therefore
\[
\frac{{\boldsymbol c}^*_{a}}{{\boldsymbol c}^*_{a_1}} = 
-\frac{ \mu_{l'}-\mu_{l'-1}}{\mu_{l'}-\lambda}.
\]
Because $l' > 2$ and $l'-1 \ge 2$, definition \eqref{e:defalphastar} 
implies $\alpha^*(l') =1$, $\alpha^*(l'+1) = \rho_l$,
$\alpha^*_1(l) = \rho_{l-1}$ and $\alpha^*_1(l+1) = \rho_l$.
These and \eqref{e:Cjbetaalphatandem} imply
\[
\frac{C(l,\rho_{\boldsymbol d}, \alpha^*_{a_1})}{C(l,\rho_{\boldsymbol d},\alpha^*_a)} = \frac{ \mu_{l'} - \mu_{l'-1}}{\mu_{l'} - \lambda}.
\]
The last two display imply \eqref{e:conditiononCs} for $a(1) = l'$. 

If $l' > a(1)$, let $j>1$ be the position of $l$ in $a$, i.e., $l=a(j)$.  
In this case, the definition \eqref{e:solutiontandem} implies that
the running products in the definitions of 
${\boldsymbol c}^*_{a}$ and ${\boldsymbol c}^*_{a_1}$ 
are a product of the same ratios
except for the $(l')^{th}$ terms, 
which is $(\mu_{l'} - \lambda)/(\mu_{l'} -\mu_{a(j-1)})$ 
for $a$ and $(\mu_{l'} - \lambda)/(\mu_{l'} -\mu_{l'-1})$ for $a_1$. 
$a_1$ has
one more element than $a$, therefore, the ratio of the parity terms 
is again $-1$; these imply
\[
\frac{{\boldsymbol c}^*_{a}}{{\boldsymbol c}^*_{a_1}} = 
-\frac{ \mu_{l'}-\mu_{l'-1}}{\mu_{l'}-\mu_{a(j-1)}}.
\]
On the other hand, $l' \in a$,  $j > 1$,  
$a_1 = a \cup \{l'-1\}$ 
and the definition \eqref{e:defalphastar} imply 
$\alpha^*(l') = \rho^*(a(j-1))$, $\alpha^*(l'+1) = \rho_{l'}$, 
$\alpha_1^*(l') = \rho_{l'-1}$,
and $\alpha_1^*(l'+1) = \rho_{l'}$ and therefore 
\[
\frac{C(l',\rho_{\boldsymbol d}, \alpha^*_{a_1})}{C(l',\rho_{\boldsymbol d},\alpha^*_a)} = 
\frac{ \mu_{l'} - \mu_{l'-1}}{\mu_{l'} - \mu_{a(j-1)}}.
\]
The last two displays once again imply \eqref{e:conditiononCs} 
for $l' > a(1).$

Consider a vertex $a \cup \{ {\bm d} \}$ of $G_{{\bm d}, {\bm d}}$;
by definition \eqref{e:Gdloops}, the loops on this vertex are
$\{2,3,...,{\bm d}\} - a \cup \{{\bm d}\}.$
For $l \in \{2,3,...,{\bm d}\} - a \cup \{{\bm d}\}$
the definition \eqref{e:defalphastar}
implies 
\[
\alpha^*_{a\cup\{{\bm d}\}}(l) =  \alpha^*_{a\cup\{{\bm d}\}}(l+1);
\]
we have already shown $\alpha^*_{a\cup\{d\}} \in {\mathcal H}$,
then, Lemma \ref{l:conditionforintersection} and the last display
imply
$\alpha^*_{a\cup \{\boldsymbol d\}} \in {\mathcal H}_l$ for
$l \in \{2,3,...,{\bm d}\} - a \cup \{{\bm d}\}$;
 i.e., the last part of Definition
\ref{d:Yharmonic} is also satisfied. This finishes the proof of the
proposition.
\end{proof}

\begin{proposition}\label{p:defhd}
\begin{equation}\label{e:defhd}
h^*_d \doteq \sum_{ a \subset \{1,2,3,...,d-1\}} 
{\boldsymbol c}^*_{a\cup \{d\}} [(\rho_d, \alpha^*_{a\cup\{d\}}),\cdot],
\end{equation}
$d=1,2,3,...,{\boldsymbol d}$, are $\partial B$-determined $Y$-harmonic functions.
\end{proposition}
\begin{proof}
That $h^*_d$ is $Y$-harmonic follows from Propositions \ref{p:provesolution} 
and \ref{p:simpleharmonicfunctions}.
The components of $\alpha^*_{a\cup\{d\}}$, $a\subset
\{1,2,3,...,d-1\}$ and $\beta^*_d = \rho_d$ are all between $0$ and $1$.
This and Proposition \ref{p:balayagesimpled} imply 
that  $h^*_d$ are all $\partial B$-determined.
\end{proof}

With definition \eqref{e:defhd} we can rewrite \eqref{e:embeddingfunc} as
\begin{equation}\label{e:embeddingfunc1}
-\left(\prod_{l=d+1}^{{\boldsymbol d}} \frac{ \mu_l - \lambda}{\mu_l - \mu_d}\right)
h_d^*(y)
=
\sum_{ a \in V_{G_{\boldsymbol d}^d }} {\boldsymbol c}_a^* [ (\beta_a^*, \alpha_a^*), y ]
\end{equation}
for $y \in {\partial B}.$

\begin{theorem}\label{p:exactformulaDtandem}
\begin{equation}\label{e:exactformulaDtandem}
{\mathbb P}_y(\tau < \infty) = 
\sum_{d=1}^{{\boldsymbol d}}
\left( \prod_{l=d+1}^{\boldsymbol d} \frac{\mu_l -\lambda}{\mu_l - \mu_d}  \right) 
h^*_d(y)
\end{equation}
for $y \in B.$
\end{theorem}
For ${\bm d}=2$ \eqref{e:exactformulaDtandem} reduces to
\[
{\mathbb P}_y(\tau < \infty) =
\left( \rho_2^{y(1)-y(2)} 
- \frac{\mu_2-\lambda}{\mu_2 -\mu_1} \rho_2^{y(1)-y(2)} \rho_1^{y(2)}\right)
+ \frac{\mu_2 - \lambda}{\mu_2 - \mu_1} 
\rho_1^{y(1)}, 
\]
which is the formula given in \cite{sezer2018approximation} for 
${\bm d}=2$.

\begin{proof}
Let ${\boldsymbol 1} \in {\mathbb C}^{\{2,3,..,{\bm d}\}}$ denote the vector with
all components equal to $1$.
The decomposition of $G_{\boldsymbol d}$ into the single vertex ${\{\boldsymbol d}\}$
and $G_{\boldsymbol d}^d$, $d < {\boldsymbol d}$ implies that the right side of 
\eqref{e:exactformulaDtandem} equals
\begin{align*}
&[(\rho_{\boldsymbol d}, {\boldsymbol 1}),y] + 
\sum_{d=1}^{{\boldsymbol d}-1}
\sum_{a \in V_{G_{\boldsymbol d}^d}} {\boldsymbol c}^*_a
[ ( \beta_a^*, \alpha^*_a), y ].
+\sum_{d=1}^{{\boldsymbol d}-1}
\left( \prod_{l=d+1}^{\boldsymbol d} \frac{\mu_l -\lambda}{\mu_l - \mu_d}  \right) 
h^*_d(y)
\intertext{for $y \in \partial B$ \eqref{e:embeddingfunc1} implies}
&~~= 
[(\rho_{\boldsymbol d}, {\boldsymbol 1}),y],
\end{align*}
which, for $y\in \partial B$, equals $1$.
Thus, we see that
the right side of \eqref{e:exactformulaDtandem} equals $1$ on $\partial B$.
Proposition \ref{p:defhd} says that the same function is $\partial B$-determined 
and is $Y$-harmonic. Then its restriction to $B$
must be indeed equal to $y\rightarrow {\mathbb P}_y(\tau < \infty)$, $y \in B$,
which is the unique function with those properties.
\end{proof}

\section{Numerical Example}\label{s:numerical}
Take a four dimensional tandem system with rates, for example,
\[
\lambda = 1/18, \mu_1 = 3/18, \mu_2 = 7/18, \mu_3 = 2/18, \mu_4 = 5/18.
\]
For $n=60$, and in four dimensions, the probability
${\mathbb P}_x(\tau_n< \tau_0))$ can be computed numerically by
iterating the harmonic equation 
${\mathbb P}_x(\tau_n< \tau_0)=
{\mathbb E}_x[{\mathbb P}_{X_1}(\tau_n< \tau_0))]$.
Let $f(y)$ denote the right side of \eqref{e:exactformulaDtandem}.
Define
$V_n = -\log({\mathbb P}_x(\tau_n< \tau_0))/n$ and $W_n = -\log f(T_n(x))/n$.
The level curves of $V_n$ and $W_n$
and the graph of the relative error $(V-W)/V$ 
for $x= (i,j,0,0,0)$ and $x=(0,i,0,j,0)$,
$i,j \le n=60$ are shown in Figure \ref{f:relativeerror4}; 
qualitatively
these graphs show results similar to those reported in 
\cite{sezer2018approximation}:
almost zero relative error across the domain selected,
except for a boundary layer
along the $x(4)$-axis, where the relative error is bounded by $0.05.$
The size of the boundary layer is determined by the set
$R_\rho$ of \eqref{d:Rrho} and Theorem \ref{t:convergence}.

\ninseps{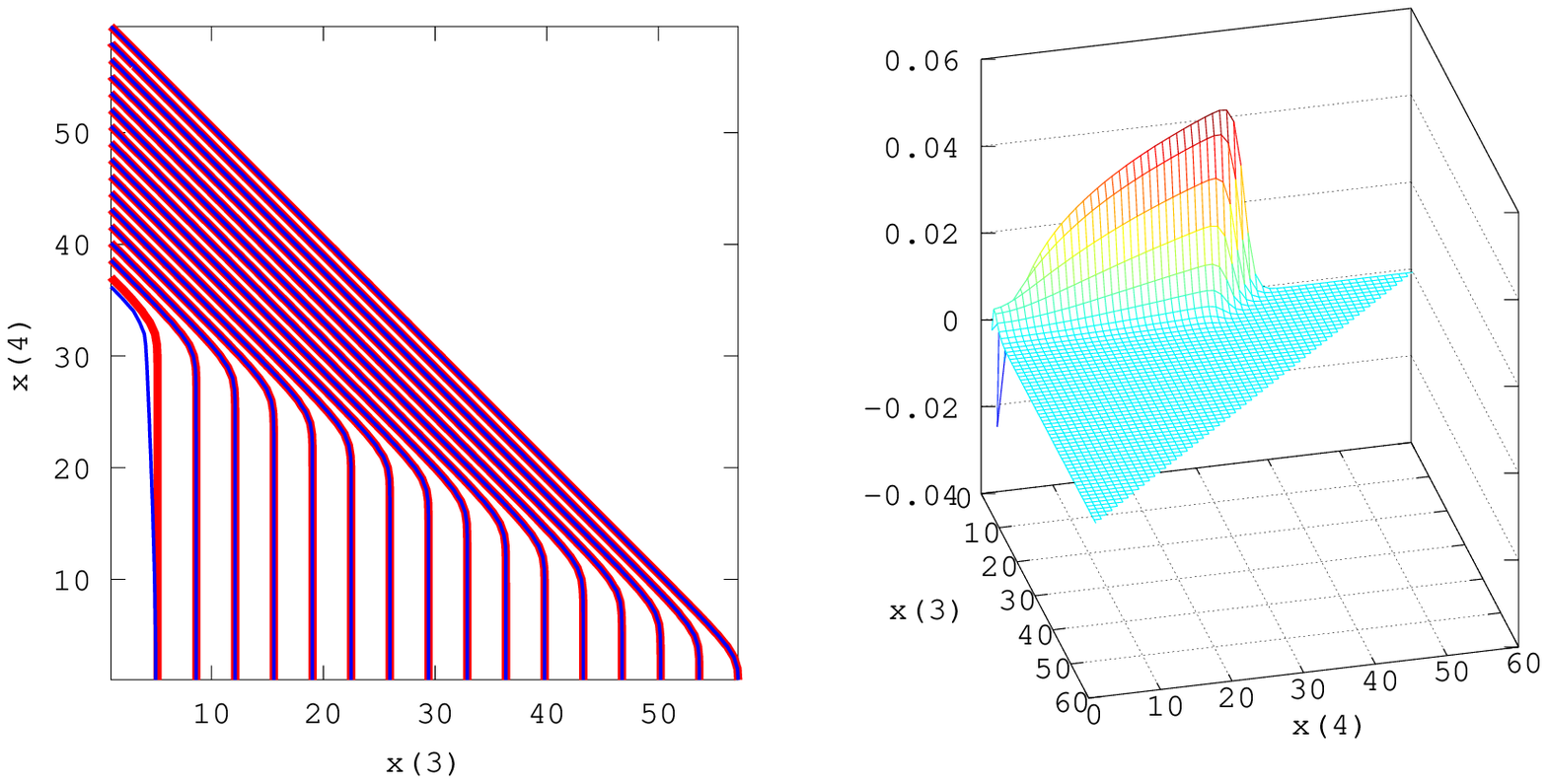}{Level curves and relative error in four dimensions}{0.65}

Finally we consider the $14$-tandem queues with parameter values shown in
Figure \ref{f:ratesex14}.

\ninsepsc{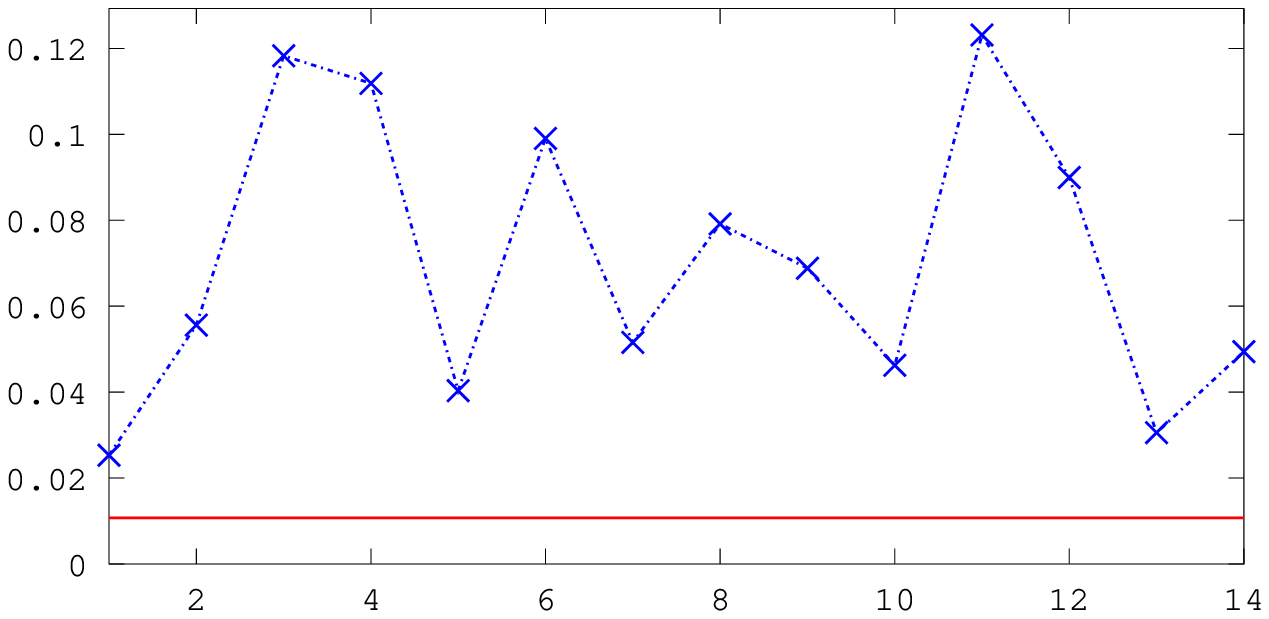}{The service rates (blue) and the arrival rate (red) 
for a $14$-dimensional tandem Jackson network}{0.8}

For $n=60$, $A_n$ contains $60^{14}/14!=  8.99 \times 10^{13}$
states which makes impractical an exact calculation via iterating 
the harmonic equation satisfied by ${\mathbb P}_y(\tau_n <\tau_0).$
On the other hand, \eqref{e:exactformulaDtandem} has $2^{15}=32768$ summands and
can be quickly calculated. Define $W_n$ as before. 
Its graph over $\{x: x(4) + x(14) \le 60, x(j) = 0, j \neq 4,14\}$
is depicted in Figure \ref{f:meshex14}. 
\ninsepsc{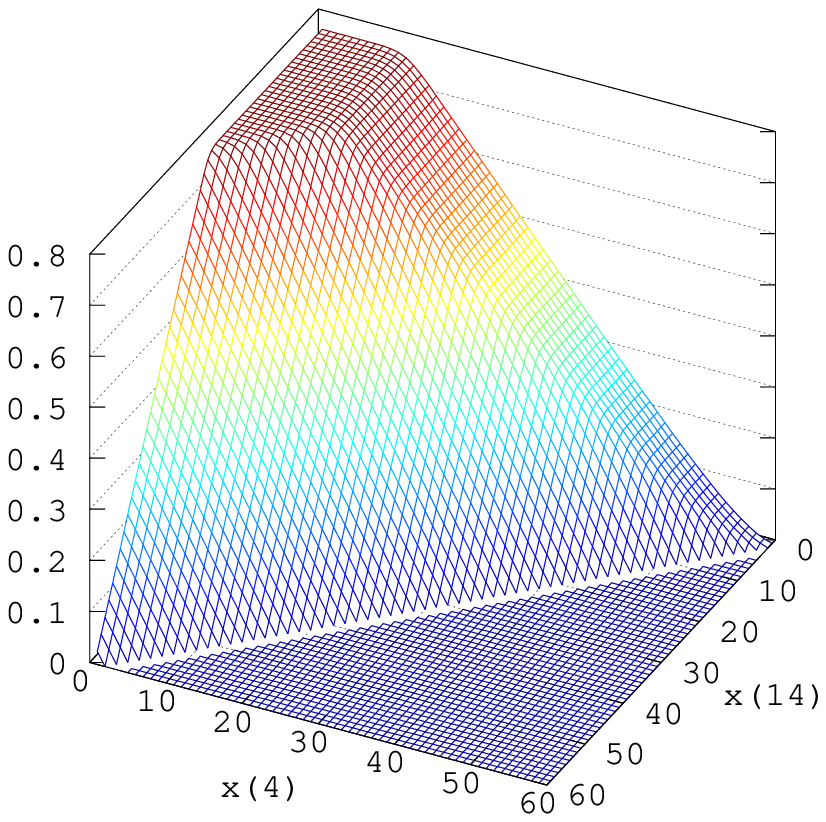}{The graph of $W_n$ over  $\{x: x(4) + x(14)= 60, x(j) = 0, j \neq 4,14\}$}{0.8}
For a finer approximation of ${\mathbb P}_{(1,0,\cdots,0)} (\tau_n < \tau_0)$ we
use importance sampling based on $W_n$. With $12000$ samples IS gives the
estimate $7.53 \times 10^{-20}$ with an estimated $95\%$ confidence interval  
$[6.57, 8.48] \times 10^{-20}$ (rounded to two significant figures). 
The value given by
our approximation \eqref{e:exactformulaDtandem}
for the same
probability is $f((1,0,\cdots,0)) = 1.77 \times 10^{-20}$ 
which is approximately $1/4^{th}$ of the estimate given
by IS. 
The large deviation estimate of the same probability is
$(\lambda/\min_{i=1}^{14}(\mu_i))^{60} = 4.15\times10^{-23}$.
The discrepancy between IS and \eqref{e:exactformulaDtandem} 
 quickly disappears as $x(1)$ increases. 
For example, for
$x(1) = 4$ , IS gives $2.47\times 10^{-19}$ and \eqref{e:exactformulaDtandem}
gives $2.32 \times 10^{-19}$.

\section{Conclusion}\label{s:conclusion}
In Section \ref{s:hstq} we computed ${\mathbb P}_y(\tau < \infty)$
under the assumption $\mu_i \neq \mu_j$ for $i\neq j.$
One can obtain formulas for ${\mathbb P}_y(\tau < \infty)$ when
this assumption is violated by computing limits
of \eqref{e:exactformulaDtandem} as $\mu_i \rightarrow \mu_j$; this
limiting process introduces polynomial terms to the formula. 
For example, for $d=3$ and
$\mu_1 = \mu_2 = \mu_3=\mu$ we get
\[
{\mathbb P}_y(\tau < \infty) 
= \rho^{\bar{y}(1)}
\left(
\frac{1}{2}c_0^2 (\bar{y}(1))^2 \rho^{y(2)+y(3)}
+  \rho^{y(3)} \left( \left( \frac{c_0^2}{2} + y(3) c_0^2 \right) 
\rho^{y(2)} + c_0\right) \bar{y}(1) +1 \right),
\]
where $c_0 = (\mu - \lambda)/\mu$
and $\bar{y}(1) = y(1) -( y(2) + y(3))$.
Similar limits can be computed explicitly for the cases
$\mu_1 = \mu_2 \neq \mu_3$, $\mu_1=\mu_3 \neq \mu_2$ and
$\mu_1 \neq \mu_2  = \mu_3$.
A systematic study of these cases in three and higher dimensions 
remain for future work.

Another obvious direction for future work is the study of more
general dynamics and exit boundaries. We refer the reader
to \cite[Conclusion]{sezer2018approximation} for further comments
on possible directions for future research.

\bibliography{balayage}

\end{document}